\documentclass[12pt]{amsart}
\usepackage{amssymb}
\usepackage{amscd}
\usepackage{amsmath}
\usepackage{amsfonts}
\usepackage{latexsym}
\usepackage{hyperref}
\usepackage{xcolor}
\usepackage[all]{xy}
\usepackage{verbatim}
\usepackage{tabularx}
\usepackage{array}
\usepackage{booktabs}
\usepackage[utf8]{inputenc} 
\usepackage[margin=1in]{geometry}

\DeclareFontFamily{U}{mathc}{}
\DeclareFontShape{U}{mathc}{m}{it}%
{<->s*[1.03] mathc10}{}
\DeclareMathAlphabet{\mathcal}{U}{mathc}{m}{it}

\newtheorem{theorem}{Theorem}[section]
\newtheorem{lemma}[theorem]{Lemma}
\newtheorem{proposition}[theorem]{Proposition}
\newtheorem{corollary}[theorem]{Corollary} 
\theoremstyle{definition}  
\newtheorem{definition}[theorem]{Definition}
\newtheorem{example}[theorem]{Example}
\newtheorem{conjecture}[theorem]{Conjecture}  

\newtheorem{remark}[theorem]{Remark}

\definecolor{dmitri}{HTML}{AF72B0}
\newcommand{\dmitri}[1]{\textcolor{dmitri}{[[Dmitri: #1]]}}

\newcommand{\id}{\text{id}}

\newcommand{\Fun}{\text{Fun}}
\newcommand{\FPdim}{\text{FPdim}} 

\renewcommand{\Vec}{\operatorname{\operatorname{\mathsf{Vec}}}}

\DeclareMathOperator{\Aut}{\operatorname{\mathsf{Aut}}}

\DeclareMathOperator{\Inv}{\operatorname{\mathsf{Inv}}}

\DeclareMathOperator{\Ind}{\operatorname{\mathsf{Ind}}}
\DeclareMathOperator{\Hom}{\operatorname{\mathsf{Hom}}}
\DeclareMathOperator{\Alt}{\operatorname{\mathsf{Alt}}}

\newcommand{\ad}{\text{ad}}

\newcommand{\B}{\mathcal{B}}
\newcommand{\C}{\mathcal{C}}
\newcommand{\D}{\mathcal{D}}
\newcommand{\E}{\mathcal{E}}
\newcommand{\F}{\mathcal{F}}
\newcommand{\Z}{\mathcal{Z}}

\newcommand{\M}{\mathcal{M}}
\newcommand{\A}{\mathcal{A}}
\newcommand{\T}{\mathcal{T}}

\newcommand{\N}{\mathcal{N}}

\newcommand{\TY}{\mathcal{T}\mathcal{Y}}

\newcommand{\be}{\mathbf{1}}

\renewcommand{\be}{\mathbf{1}}

\newcommand{\bt}{\boxtimes}
\newcommand{\ot}{\otimes}

\newcommand{\MT}{T}

\newcommand{\beq}{\begin{equation}}
\newcommand{\eeq}{\end{equation}}

\newcommand{\bpf}{\begin{proof}}
\newcommand{\epf}{\end{proof}}

\newcommand{\bth}{\begin{theorem}}
\renewcommand{\eth}{\end{theorem}}
\newcommand{\bpr}{\begin{proposition}}
\newcommand{\epr}{\end{proposition}}
\newcommand{\ble}{\begin{lemma}}
\newcommand{\ele}{\end{lemma}}
\newcommand{\bco}{\begin{corollary}}
\newcommand{\eco}{\end{corollary}}
\newcommand{\bde}{\begin{definition}}
\newcommand{\ede}{\end{definition}}
\newcommand{\bex}{\begin{example}}
\newcommand{\eex}{\end{example}}
\newcommand{\bre}{\begin{remark}}
\newcommand{\ere}{\end{remark}}
\newcommand{\bcj}{\begin{conjecture}}
\newcommand{\ecj}{\end{conjecture}}

\newcommand{\Mod}{\bf{Mod}}

\newcommand{\Id}{\mathcal{I}\it{d}}
\newcommand{\sVect}{\mathcal{sV}  \mkern-3mu ect} 
\newcommand{\Rad}{\mathcal{R \mkern-3mu ad}}
\newcommand{\Corad}{\mathcal{C \mkern-3mu orad}}
\newcommand{\Rep}{\mathcal{R \mkern-3mu ep}}
\newcommand{\Core}{\mathcal{C \mkern-3mu ore}}
\newcommand{\Mantle}{\mathcal{M \mkern-3mu antle}}
\newcommand{\Vect}{\mathcal{V \mkern-3mu ect}}

\newcommand{\Triv}{\mathcal{T\mkern-3mu riv}}
\newcommand{\Mor}{\mathcal{M\mkern-3mu or}}
\newcommand{\Gaug}{\mathcal{G}}

\newcommand{\uuPic}{\mathbf{Pic}}
\newcommand{\uAut}{\mathcal{A\mkern-3mu ut}}

\begin{document}
\title[Tannakian Radical and Mantle]{The Tannakian radical and the  mantle of a braided fusion category}
\author{Jason Green}
\address{Department of Mathematics and Statistics,
University of New Hampshire,  Durham, NH 03824, USA}
\email{jason.green@unh.edu}
\author{Dmitri Nikshych}
\address{Department of Mathematics and Statistics,
University of New Hampshire,  Durham, NH 03824, USA}
\email{dmitri.nikshych@unh.edu}

\begin{abstract}
We define the Tannakian radical of a braided fusion category $\C$ as the intersection of its maximal Tannakian subcategories. The localization of $\C$ corresponding to the Tannakian radical, termed the  mantle of $\C$, admits a canonical central extension that serves as a complete invariant of $\C$. The mantle has a trivial Tannakian radical, and we refer to braided fusion categories with this property as reductive. We investigate the properties and structure of reductive categories and prove several classification results.
\end{abstract}

\maketitle
\baselineskip=18pt
\tableofcontents


\section{Introduction}

Let $\B$ be a braided fusion category.
There is a well-known procedure for localizing $\B$ with respect to an 
\'etale algebra
(i.e., a commutative separable algebra) $A$ in $\B$, consisting of taking the category
$\B_A^{loc}$ of local $A$-modules in $\B$, see, e.g., \cite{DMNO, DNO}. This produces
a ``smaller" braided fusion category in the Witt class of $\B$.  In the special case
where $A$ is the regular algebra in a Tannakian subcategory $\E \subset \B$, this procedure
coincides with the de-equivariantization \cite{DGNO2}: $\B_A^{loc} \cong \E' \bt_\E \Vect$,
where $\E'$ is the centralizer of $\E$. When $\E$ is maximal among 
Tannakian subcategories of $\B$, this localization is an invariant of $\B$, 
called the {\em core}. 
It was explained in \cite{DGNO2} how the core can be used to extract the part of $\C$ 
that does not come from finite groups. This invariant  has turned out to be very useful in the classification of fusion categories. For example, non-degenerate braided fusion categories with a trivial core are precisely representation categories of twisted Drinfeld doubles \cite{DGNO2}. In addition, Natale characterized weakly group-theoretical braided fusion categories in terms of their cores \cite{Na}.

The category $\B$ can be recovered from its core by means
of a gauging procedure \cite{CGPW, ENO2} with respect to a group $G$ such that $\E=\Rep(G)\subset \B $ is a maximal Tannakian subcategory. However, unlike the core, the isomorphism class of $G$ is not an invariant of $\B$, see Remark~\ref{not an invariant}. Thus, it is natural to replace $\E$ with an invariant Tannakian subcategory of $\B$ and investigate the corresponding localization. This is what we do in this paper, whose main results and organization are outlined below.

Sections~\ref{Sect prelim} and \ref{Section: gauge} contain the necessary results about gauging and localization of braided fusion categories and their maximal Tannakian subcategories.

We define the {\em Tannakian radical} $\Rad(\B)$ of $\B$ as the intersection of all maximal Tannakian subcategories of $\B$ (Section~\ref{section Rad(B)}). We say that $\B$ is {\em reductive} if $\Rad(\B)$ is trivial, i.e., $\Rad(\B)=\Vect$. This is a new class of braided fusion categories 
with interesting properties. For example, $\Rep(D(G))$, the representation category of the Drinfeld double of a finite group $G$, is reductive if and only if $G$ is generated by its normal Abelian subgroups, see Remark~\ref{rem nilpotency} ($G$ must be nilpotent in this case).

We introduce the {\em mantle} of $\B$ as its  localization with respect to the radical:
\[
\Mantle(\B)= \Rad(\B)'\bt_{\Rad(\B)}\Vect. 
\]
We show that the mantle is reductive (Corollary~\ref{Mantle reductive}).
Clearly, the braided equivalence classes of $\Rad(\B)$ and $\Mantle(\B)$ are invariants of $\B$. 
Furthermore, there is a group $G_\B$, determined up to isomorphism, such that $\Rad(\B)= \Rep(G_\B)$ and a canonical monoidal $2$-functor $F_\B: G_\B \to \uuPic(\Mantle(\B))$ such that $\B$ is
equivalent to the equivariantization of the corresponding central extension of $\Mantle(\B)$.
In other words, $\B$ is a canonical $G_\B$-gauging of its mantle. These data
$(G_\B,\, \Mantle(\B),\, F_\B)$ constitute a complete invariant of $\B$
(Section~\ref{Sect complete invrariant}). Since the center $\Z(\A)$ of a fusion category $\A$ is a complete invariant of its Morita equivalence class, by considering the radical and mantle of $\Z(\A)$
we obtain new Morita invariants of $\A$ (Corollary~\ref{complete Morita invariant}).

We study the structure of reductive categories in detail in Section~\ref{Sect  structure reductive}. 
We define a {\em coradical}
of a braided fusion category as the subcategory generated by all its Tannakian subcategories.
The coradical is a group-theoretical category (Proposition~\ref {saturated WGT}). 
We classify reductive categories into three types based on the non-degeneracy properties of the categories themselves and their coradicals, see Definition~\ref{types of decomposition}.
In Theorem~\ref{canonical decomposition} we provide a decomposition of a reductive braided fusion category of each type in terms of its coradical  and a certain anisotropic subcategory. 

Pointed reductive braided fusion categories (i.e., pre-metric groups) are studied in Section~\ref{Sect reductive pointed}. Their classification is given in Theorems~\ref{pointed saturated} and \ref{reductive metric 2-groups classification}.  

Section~\ref{sect classification results}  is devoted to classification results. We first deal with reductive $p$-categories (i.e., categories whose Frobenius-Perron dimension is a power of a prime integer $p$). The motivation comes from Proposition~\ref{red WGT nil}, which states
that a reductive braided fusion category is weakly group-theoretical if and only if it is nilpotent, and so it is a direct product of $p$-categories by \cite{DGNO1}. We show that all such categories of dimension $\leq p^5$ are pointed and 
we classify the non-pointed reductive 
$p$-categories of dimension $p^6$ in Proposition~\ref{Prop p+1 classes}. Finally, we apply our techniques to the classification of non-degenerate braided fusion categories of small Frobenius-Perron dimension. 
In particular, we provide a complete list of categories of dimension equal to a product of at most four primes in Example~\ref{prop: list}, and we relate our results to those already present in the literature.
 
{\bf Acknowledgments.} 
We are grateful to Alexei Davydov and Pavel Etingof for illuminating discussions and to the anonymous referee whose comments helped us improve 
the exposition. The work of the second author was supported  by  the  National  Science  Foundation  under  
Grant No.\ DMS-2302267. This material is based upon work supported by the National Science Foundation under Grant No.\ DMS-1928930, while the second author was in residence at SLMath (the Mathematical Sciences Research Institute) in Berkeley, California, during the Summer of 2024.


\section{Preliminaries}
\label{Sect prelim}

Throughout this paper we work over an algebraically closed field $k$ of characteristic $0$.
Tensor and module categories are assumed to be finite semisimple and $k$-linear. 

We refer the reader to \cite{EGNO, DGNO1} for the theory of braided fusion categories and to \cite{DN2, ENO1} for the theory of graded extensions.

\subsection{Gauging and localization of braided fusion categories}

Let $\B$ be a braided fusion category. Let $\uuPic(\B)$ denote the 2-categorical Picard group of $\B$, consisting of invertible $\B$-module categories \cite{ENO2}.

Let $G$ be a finite group.
A {\em central $G$-extension} $\C=\bigoplus_{x\in G}\,\B_x$ of $\B_1=\B$ (which is the same thing as a $G$-crossed extension \cite{DN2, JMPP}) is a faithful $G$-graded extension along with an embedding $\B \hookrightarrow \Z(\C)$ whose composition with the forgetful functor $\Z(\C)\to \C$  coincides with the embedding $\B \hookrightarrow \C$. 
It was shown in \cite{ENO2, DN2} that the $2$-groupoid of central extensions of $\B$ is equivalent to the $2$-groupoid of monoidal $2$-functors $F: G \to \uuPic(\B)$.
A central extension carries a canonical action of $G$ such that $g(\B_x)=\B_{gxg^{-1}},\, g,x\in G$. This restricts to an action of $G$ on $\B$
by braided autoequivalences that coincides with the composition of $F$ with the canonical monoidal functor $\uuPic(\B) \to \uAut^{br}(\B)$.

The corresponding equivariantization 
\[
\Gaug(G,\B,F) :=\left( \bigoplus_{x\in G}\,\B_x \right)^G
\]
is called a $G$-{\em gauging} (or simply {\em gauging}) of $\B$. It is a  braided fusion category containing a Tannakian subcategory $\Rep(G)$. This category is non-degenerate 
if and only if $\B$ is non-degenerate.

Conversely, to a non-degenerate braided fusion category $\C$ containing a Tannakian subcategory $\E\cong \Rep(G)$, one associates a braided fusion category $\B=\E'\bt_\E \Vect$, 
where $\E'$ is the centralizer of $\E$ in $\C$. This category $\B$ is called a {\em localization} of $\C$ (with respect to $\E$). It is equipped  with  a monoidal $2$-functor $F: G \to \uuPic(\B)$ such that $\C$ is canonically equivalent to $\Gaug(G,\B,F)$.

\subsection{Maximal Tannakian subcategories and the core of a braided fusion category}
\label{Section: max Tannakian}

Let $\B$ be a braided fusion category. 
Let $T(\B),\,T_{max}(\B)$ denote the sets of all Tannakian and maximal Tannakian subcategories of $\B$, respectively. If $G$ is a group acting on $\B$ by braided autoequivalenves,  let $T^G(\B),\,T_{max}^G(\B)$
denote the sets of all $G$-stable Tannakian and maximal $G$-stable Tannakian subcategories of $\B$, respectively.

A Tannakian subcategory $\E$ of a non-degenerate braided fusion category  $\B$ 
is called {\em Lagrangian} if $\E=\E'$, where $\E'$ is the centralizer of $\E$ in $\B$.  
Lagrangian subcategories of $\B$ are in bijection with isomorphism classes
of braided equivalences between $\B$ and centers of pointed fusion categories \cite{DGNO2}. These centers are precisely the representation categories of twisted Drinfeld doubles of finite groups, $\Rep(D^\omega(G))\cong \Z(\Vect_G^\omega)$, where $G$ is a finite group and $\omega$ is a $3$-cocycle on $G$.
Let ${L}(\B)$ denote the set of Lagrangian subcategories of $\B$. 

\begin{definition}
\label{def metabolic}
A non-degenerate category $\B$ with ${L}(\B)\neq \emptyset$ is called {\em metabolic}. 
\end{definition}

Since
a Lagrangian subcategory is necessarily maximal, for metabolic categories one has $L(\B)= T_{max}(\B)$. 

It was shown in \cite{DGNO2} that for any $\E=\Rep(G_\E) \in T_{max}(\B)$ the braided equivalence class of the corresponding localization $\E'\bt_\E \Vect$ does not depend on the choice of $\E$. This category is called the {\em core} of $\B$
and is denoted $\Core(\B)$.  Furthermore, the image of $G_\E$ in $\uAut^{br}(\Core(\B))$ is also independent of the choice of $\E$. 

\begin{remark}
\label{not an invariant}
We saw that $\B$ can be realized as a gauging of 
a canonical category $\Core(\B)$. But this gauging 
is not canonical, since the above group $G_\E$ is not an invariant of $\B$. For example, $\Rep(D(\mathbb{Z}/4\mathbb{Z}))$ contains
maximal Tannakian subcategories $\Rep(\mathbb{Z}/4\mathbb{Z})$ and $\Rep(\mathbb{Z}/2\mathbb{Z} \times \mathbb{Z}/2\mathbb{Z})$ and so it is both a $\mathbb{Z}/4\mathbb{Z}$- and $\mathbb{Z}/2\mathbb{Z} \times \mathbb{Z}/2\mathbb{Z}$-gauging of the trivial category $\Vect$. In other words,
$\Rep(D(\mathbb{Z}/4\mathbb{Z})) \cong \Rep(D^\omega(\mathbb{Z}/2\mathbb{Z} \times \mathbb{Z}/2\mathbb{Z}))$ for some $3$-cocycle $\omega$ on $\mathbb{Z}/2\mathbb{Z} \times \mathbb{Z}/2\mathbb{Z}$.  Braided equivalence between twisted Drinfeld doubles were studied  by many authors, including \cite{GMN, D, N, U, MS}.

\end{remark}

\subsection{Metric groups}
\label{sect metric groups}

Recall that a {\em pre-metric group} $(A,\,q)$ is a finite Abelian group $A$ equipped with a quadratic form $q:A \to k^\times$. There is an equivalence between the groupoid of pre-metric groups and that of pointed braided fusion categories
\[
(A,\,q) \mapsto \C(A,\,q),
\]
where $A$ is identified with the group of invertible objects of $\C(A,\,q)$ and the scalar $q(x)$ is the value of the self-braiding
of $x\in A$.

A pre-metric group $(A,\,q)$ is called a {\em metric group}
if the quadratic form $q$ is non-degenerate. This is equivalent to $\C(A,\,q)$ being a non-degenerate braided fusion category.

A pre-metric group $(A,\,q)$ is {\em anisotropic} if $q(x)\neq 1$
for any $x\neq 0$.

For any prime $p$, let $M_p$ be the monoid of metric $p$-groups. This monoid was described in \cite{W} and \cite{KK} by generators and relations as follows. 

Suppose that $p$ is odd. For any $p^m$-th root of unity $\zeta\in k^\times$ define a 
quadratic form
\begin{eqnarray}
\label{qzetaodd}
& q_\zeta: \mathbb{Z}/{p^m}\mathbb{Z} \to k^\times,\qquad & q_\zeta(x)= \zeta^{x^2}.
\end{eqnarray}
This form is non-degenerate if and only if $\zeta$ is a primitive $p^m$-th root of $1$.
When $\eta$  is another such root,
the metric groups $(\mathbb{Z}/{p^m}\mathbb{Z},\, q_\zeta)$ and $(\mathbb{Z}/{p^m}\mathbb{Z},\, q_{\eta})$
are isomorphic if and only if $\eta=\zeta^{a^2}$ for some $a$
relatively prime to $p$. 

Fix a quadratic non-residue $b$ modulo $p^m$ and let $\bar{\zeta} = \zeta^b$. The isomorphism 
classes of metric groups 
$(\mathbb{Z}/{p^m}\mathbb{Z},\, q_\zeta)$ and $(\mathbb{Z}/{p^m}\mathbb{Z},\, q_{\bar{\zeta}}),\,
m\geq 1$, generate $M_p$.  These generators obey the relations
\begin{equation}
\label{eqn relations p odd}
(\mathbb{Z}/{p^m}\mathbb{Z},\, q_\zeta) \oplus (\mathbb{Z}/{p^m}\mathbb{Z},\, q_\zeta)
\cong (\mathbb{Z}/{p^m}\mathbb{Z},\, q_{\bar{\zeta}}) \oplus (\mathbb{Z}/{p^m}\mathbb{Z},\, q_{\bar{\zeta}})
,\qquad m \geq 1.
\end{equation}

Consider the hyperbolic quadratic form 
\[
h : \mathbb{Z}/{p^m}\mathbb{Z} \times \mathbb{Z}/{p^m}\mathbb{Z}
\to k^\times,\qquad h(x,\,y)=\zeta^{xy}.
\]
The corresponding metric group $(\mathbb{Z}/{p^m}\mathbb{Z} \times \mathbb{Z}/{p^m}\mathbb{Z},\,h)$ is
isomorphic to $(\mathbb{Z}/{p^m}\mathbb{Z},\, q_\zeta) \oplus (\mathbb{Z}/{p^m}\mathbb{Z},\, q_\zeta)$ if 
$p\equiv 1 (\mod 4)$ and
to $(\mathbb{Z}/{p^m}\mathbb{Z},\, q_\zeta) \oplus (\mathbb{Z}/{p^m}\mathbb{Z},\, q_{\bar{\zeta}})$
if $p\equiv 3 (\mod 4)$.

The classification of anisotropic metric $p$-groups is well known. We recall it in Table~\ref{table-1}.

\begin{table}[h!]
\centering
\begin{tabular}{|p{0.3cm}|p{3cm}|p{7.2cm}|p{1.7cm}|}
 \hline
 & Odd prime $p$ & Metric group & \tiny{number~of iso~classes}  \\
\hline
\hline
$1.$ & \text{any} & $(\mathbb{Z}/{p}\mathbb{Z},\, q_{\zeta})$ &   2 \\
\hline
$2.$ & $p\equiv 1(\mod 4)$ &  $(\mathbb{Z}/{p}\mathbb{Z},\, q_{\zeta}) \oplus (\mathbb{Z}/{p}\mathbb{Z},\, 
q_{\bar{\zeta}})$
& 1\\
\hline
$3.$ & $p\equiv 3(\mod 4)$ &  $(\mathbb{Z}/{p}\mathbb{Z},\, q_{\zeta}) \oplus (\mathbb{Z}/{p}\mathbb{Z},\, q_{\zeta})$  & 1\\
\hline
\end{tabular}
\medskip
\caption{\label{table-1} Anisotropic metric $p$-groups (for odd $p$). 
Here $\zeta$ is a primitive $p$th root of $1$.}
\end{table}

A description of $M_2$ is more involved.
For any $2^{m+1}$-th root of unity $\xi$ and a $2^m$-th root of unity $\zeta$, $m\geq 1$, define quadratic forms
\begin{eqnarray}
\label{qzeta}
& q_\xi: \mathbb{Z}/{2^{m}}\mathbb{Z} \to k^\times,\qquad & q_\xi(x)= \xi^{x^2}, \\
\label{hzeta}
& h: \mathbb{Z}/{2^{m}}\mathbb{Z} \times \mathbb{Z}/{2^{m}}\mathbb{Z} \to k^\times,\qquad
& h(x,\,y)= \zeta^{xy},\\
\label{fzeta}
& f: \mathbb{Z}/{2^{m}}\mathbb{Z} \times \mathbb{Z}/{2^{m}}\mathbb{Z} \to k^\times,\qquad
& f(x,\,y)= \zeta^{x^2+xy+y^2}.
\end{eqnarray}
The quadratic form \eqref{qzeta} (respectively,  \eqref{hzeta}, \eqref{fzeta}) is non-degenerate if and only if $\xi$  is a primitive
$2^{m+1}$-th root of unity (respectively,
$\zeta$ is a primitive $2^m$-th root of unity). 
There are $2$ isomorphism classes of metric groups $(\mathbb{Z}/{2}\mathbb{Z},\, q_{i}),\, i^2=-1$,
and $4$ isomorphism classes of metric groups $(\mathbb{Z}/2^{m}\mathbb{Z},\, q_{\xi})$ for each $m\geq 2$.
Along with $(\mathbb{Z}/{2^{m}}\mathbb{Z} \times \mathbb{Z}/{2^{m}}\mathbb{Z},\, h)$
and $(\mathbb{Z}/{2^{m}}\mathbb{Z} \times \mathbb{Z}/{2^{m}}\mathbb{Z},\, f),\, m\geq 1,$ they generate $M_2$.
(Note that the isomorphism types of the last two metric groups are independent of the choice of $\zeta$.) 
A complete list of relations between these generators is quite long and can be found in \cite{KK, Mir}. We only note the relation
\begin{equation}
\label{EF-relation}
\left( \mathbb{Z}/{2^{m}}\mathbb{Z} \times \mathbb{Z}/{2^{m}}\mathbb{Z},\, h \right)^{\oplus 2}
\cong
 \left( \mathbb{Z}/{2^{m}}\mathbb{Z} \times \mathbb{Z}/{2^{m}}\mathbb{Z},\, f\right)^{\oplus 2},\,
 m\geq 1.
\end{equation}

\begin{lemma}
\label{qn=1 homogeneous}
A metric group $((\mathbb{Z}/2^m\mathbb{Z})^k,\,q)$ satisfies
$q^{2^m}=1$ if and only if it is a direct sum of some number of metric groups
\eqref{hzeta} and \eqref{fzeta}.
\end{lemma}
\begin{proof}
Such a metric group is isomorphic to an orthogonal sum of forms
\eqref{qzeta}, \eqref{hzeta}, and \eqref{fzeta}, but the condition
$q^{2^m}=1$   holds only for the last two.
\end{proof}

\begin{definition}
\label{def fermion}
A {\em fermion} in a pre-metric group $(A,\,q)$  is an element $x\in A$ of order $2$ such that $q(x)=-1$.  
\end{definition} 

A proof of the classification of anisotropic pre-metric $2$-groups given in Table~\ref{table-2}  can be found, e.g., in \cite[Appendix A]{DGNO2}.

\begin{table}[h!]
\centering
\begin{tabular}{|p{0.5cm}|p{5.7cm}||p{5.5cm}|p{1.7cm}| }
 \hline
  & Metric groups & Values of $q$ & \tiny{number~of iso~classes}  \\
\hline
\hline
$1^*$ & $(1,\, q_{1})$ & $1$ &   1 \\
\hline
$2^*$ & $(\mathbb{Z}/{2}\mathbb{Z},\, q_{i})$ & $1,\, i$ &   2 \\
\hline
$3^*$ & $(\mathbb{Z}/2\mathbb{Z},\, q_{i}) \oplus (\mathbb{Z}/{2}\mathbb{Z},\, q_{i})$  &  $1,\,-1,\,i,\,i$ & 2   \\
\hline
$4^*$ & $(\mathbb{Z}/{4}\mathbb{Z},\, q_{\xi})$ & $1,\,-1,\,\xi,\,\xi$  & 4 \\
\hline
$5^*$ & $(\mathbb{Z}/{4}\mathbb{Z},\, q_{\xi}) \oplus (\mathbb{Z}/{2}\mathbb{Z},\, q_{i})$ & 
$  1,\,-1,\,i,\,-i,\, \xi,\,-\xi,\,i\xi ,\,i\xi $  & 4\\
\hline
$6$ & $(\mathbb{Z}/{2}\mathbb{Z}\times \mathbb{Z}/{2}\mathbb{Z},\, f)$    & $1,\,-1,\,-1,\,-1$  & 1\\
\hline
$7$ &  $(\mathbb{Z}/{2}\mathbb{Z}\times \mathbb{Z}/{2}\mathbb{Z},\, f) \oplus (\mathbb{Z}/{2}\mathbb{Z},\, q_{i})$   &  $ 1,\,-1,\,-1,\,-1,\, i,\,-i,\,-i,\,-i $ & 2   \\
\hline
\end{tabular}
\begin{tabular}{|p{0.5cm}|p{5.7cm}||p{5.5cm}|p{1.7cm}| }
 \hline
 & Pre-metric groups &  &    \\
\hline
\hline
$8^*$ & $(\mathbb{Z}/{2}\mathbb{Z},\, q_{-1})$ & $1,\, -1$ &   1 \\
\hline
$9^*$ & $(\mathbb{Z}/{2}\mathbb{Z},\, q_{-1}) \oplus (\mathbb{Z}/{2}\mathbb{Z},\, q_{i}) $ & $1,\, -1,\, i,\, -i$ &   1 \\
\hline
\end{tabular}
\medskip
\caption{\label{table-2}Anisotropic pre-metric $2$-groups. Here $i^2=\xi^4=-1$. Lines labelled by $*$ contain
pre-metric groups with at most 1 fermion.}
\end{table}
\medskip

\section{Maximal Tannakian subcategories of a gauged category}
\label{Section: gauge}

\subsection{Subcategories of an equivariantization}


Let $T: G \to \uAut^\ot(\A)$ be a monoidal functor, i.e., an action of a group $G$ on a fusion category $\A$. Let $N\subset G$ be a normal subgroup of $G$ and let $\eta: T|_N \to {I}$ be a natural monoidal isomorphism between the restriction of $T$ to $N$ and the trivial monoidal functor $I: N \to \uAut^\ot(\A)$. We will call such $\eta$ a {\em trivialization} of $T$ on $N$. It provides a section $\A \to \A^N$ of the forgetful functor $\A^N \to \A$.

Following \cite{GJ}, we say that $\eta$ is {\em $G$-equivariant} if $\eta_{gng^{-1}} = T_g\eta_nT_{g^{-1}}$ for all $n\in N,\, g\in G$, where we identify $T_gT_nT_{g^{-1}}$ with $T_{gng^{-1}}$  using the monoidal functor structure of $T$.

One can check that $G$-equivariant trivializations of $T|_N$ are in bijection with isomorphism classes of factorizations of $T: G \to \uAut^\ot(\A)$ as a composition $G \to G/N \xrightarrow{\tilde{T}} \uAut^\ot(\A)$, where $\tilde{T}$ is a monoidal functor.

The following theorem was established by Galindo and Jones in \cite{GJ}

\begin{theorem}
\label{GJ theorem}
Let $\A$ be a fusion category along with an action $T:G \to \uAut^\ot(\A)$ of a finite group $G$. Fusion subcategories of $\A^G$ are in bijection
with triples $(\D,\, N,\, \eta)$, where $\D$ is a $G$-stable fusion subcategory of $\A$, $N$ is a normal subgroup of $G$, and $\eta: T|_N \to I$ is a $G$-equinvariant trivialization. 
\end{theorem} 

\begin{remark}
\label{GJ remark}
   The subcategory of $\A^G$ corresponding to the triple $(\D,\, N,\, \eta)$ is $\D^{G/N}$, where the action of $G/N$ on $\A$ is defined using $\eta$ and the embedding of $\D^{G/N}$ into $\A^G$ is given by
\[
\D^{G/N} \hookrightarrow (\D^N)^{G/N} \cong \D^G \hookrightarrow \A^G.
\] 
\end{remark}

\subsection{Subcategories of a central extension}

Let $\B$ be a braided fusion category, let $G$ be a finite group, and let
\begin{equation}
F: G \to \uuPic(\B) : x \mapsto  \C_x
\end{equation}
be a monoidal $2$-functor.  Let 
\begin{equation}
\label{eqn extension C} 
\C:= \bigoplus_{x\in G}\, \C_x,\qquad \C_1=\B,
\end{equation}
be the corresponding central extension of $\B$.  

For a  fusion subcategory $\tilde{\C} \subset \C$ let $\tilde{\B} = \B \cap \tilde{\C}$. 
There is a subgroup $H_{\tilde{\C}}\subset G$ such that
$\tilde{\C}$ is a central  $H_{\tilde{\C}}$-graded extension  of $\tilde{\B}$:
\[
\tilde{\C} = \bigoplus_{x\in H_{\tilde{\C}}}\, \tilde{\C}_x,\qquad \tilde{\C}_1=\tilde{\B},
\]
where $\tilde{\C}_{x}\subset \C_x$ for all $x\in H_{\tilde{\C}}$. We will refer
to $\tilde{\C}$  as a {\em subextension} of $\C$.  We will say that a subextension
$\tilde{\C}\subset \C$ is {\em faithful} if $H_{\tilde{\C}}=G$.

As an abstract central extension,
$\tilde{\C}$ corresponds to a monoidal $2$-functor
\begin{equation}
\label{eqn HGtoPic tildeB} 
\tilde{F}: H_{\tilde{\C}}\to \uuPic(\tilde{\B}) : x \mapsto  \tilde{\C}_x.
\end{equation}
The tensor product $\B \bt \tilde{\C} \to \C$  gives $\B$-module
equivalences $\B  \bt_{\tilde{\B}} \tilde{\C}_x \cong \C_x$  that combine into
a pseudo-natural isomorphism of monoidal $2$-functors:
\begin{equation}
\label{psudo natural iso} 
F|_{H_{\tilde{\C}}} \cong \Ind \circ \tilde{F},  
\end{equation}
where $\Ind :\uuPic(\tilde{\B}) \to \uuPic(\tilde{\C})$
is the induction functor:
\[
\Ind(\M) = \B \bt_{\tilde{\B}} \M. 
\]
Conversely, an isomorphism \eqref{psudo natural iso} produces a subextension of $\C$:
\[
\tilde{\C} = \bigoplus_{x\in H_{\tilde{\C}}}\, \tilde{\C}_x
\subset  \bigoplus_{x\in H_{\tilde{\C}}}\, \B \bt_{\tilde{\B}} \tilde{\C}_x
\cong \bigoplus_{x\in H_{\tilde{\C}}}\, \C_x =\C.
\]

\begin{definition}
\label{def primitive functor}
We will call a monoidal $2$-functor $F: G \to \uuPic(\B)$ {\em primitive}
if it is not induced from any $F: G \to \uuPic(\tilde{\B})$
(as in \ref{psudo natural iso}), where $\tilde{\B} \subset \B$  is a proper fusion subcategory.  Similarly, we will call a central extension
\eqref{eqn extension C} {\em primitive} if it has no proper faithful subextensions.
A $2$-functor (or  a corresponding extension)  is  {\em imprimitive} if it is not primitive.    
\end{definition}

Let us describe subextensions of $\C$ in more explicit terms.  Fix a fusion subcategory $\tilde{\B}\subset \B$. Let $\tilde{\B}^{co}$ be the {\em commutator} of $\tilde{\B}$ in $\C$, i.e., the fusion subcategory of $\C$ generated by all simple objects $X\in \C$ such that $X\ot X^*\in \tilde{\B}$ \cite{GN}. This is the largest subcategory of $\C$ that is a graded extension of $\tilde{\B}$:
\begin{equation}
\label{commutator}
    \tilde{\B}^{co} = \bigoplus_{y\in U(\tilde{\B})}\, (\tilde{\B}^{co})_y,
    \qquad (\tilde{\B}^{co})_1 = \tilde{\B},
\end{equation}
where $U(\tilde{\B})$ is the grading group  of $\tilde{\B}^{co}$.
Let $N(\tilde{\B}) \subset U(\tilde{\B})$ be the normal subgroup such that
\begin{equation}
\label{tildeBco}
\tilde{\B}^{co}\cap \B = \bigoplus_{y\in N(\tilde{\B})}\, (\tilde{\B}^{co})_y
\end{equation}
and let $G(\tilde{\B}) \subset G$ be the projection of $U(\tilde{\B})$. 
We have a short exact sequence of groups:
\begin{equation}
\label{eqn se sequence for UB}
1\to N(\tilde{\B}) \to U(\tilde{\B}) \xrightarrow{\pi} G(\tilde{\B}) \to 1.
\end{equation}


Let $O_{\tilde{\B}} \in H^2(G(\tilde{\B}),\, N(\tilde{\B}))$ be the cohomology class corresponding to the short exact sequence \eqref{eqn se sequence for UB}. 

\begin{proposition}
\label{prop parameterization of subextensions}
Subextensions of  \eqref{eqn extension C} 
are parameterized by the following data:
\begin{itemize}
\item a braided fusion subcategory  $\tilde{\B}\subset \B$,
\item a subgroup $S\subset G(\tilde{\B})$, such that 
$O_{\tilde{\B}}|_S$ is cohomologically trivial, and
\item  an element of an $H^1(S,\, N(\tilde{\B}))$-torsor.
\end{itemize}
\end{proposition}
\begin{proof}
A subextension $\tilde{\C}\subset \C$ such that $\tilde{\C}\ \cap \B =\tilde{\B}$
is contained in $\tilde{\B}^{co}$, which is a $U(\tilde{\B})$-graded extension of
$\tilde{\B}$ \eqref{tildeBco}. 
If $S$ is the support of $\tilde{\C}$ in $G$ then $\pi^{-1}(S) \to S$ splits, which is equivalent to the restriction $O_{\tilde{\B}}|_S$ being cohomologically trivial.
Such splittings form a torsor over $H^1(S,\, N(\tilde{\B}))$.

Conversely, if $S$ is a subgroup of $G$ and $\sigma: S \to \pi^{-1}(S)$ is a splitting, then 
\[
\tilde{\C}_\sigma = \bigoplus_{z \in S}\, (\tilde{\B}^{co})_{\sigma(z)}
\]
is an $S$-graded subextension of $\C$. The two above constructions are clearly inverses of each other.
\end{proof}

\begin{remark}
\label{functor defined by section}
Any choice of a splitting $\sigma: S \to \pi^{-1}(S)$  determines a monoidal
$2$-functor 
\begin{equation}
\label{eqn functor FBSs}
F_{(\tilde{\B},S, \sigma)}: S \to \uuPic(\tilde{\B}).
\end{equation}
\end{remark}

\begin{remark}
\label{G-stable subextension}
Subextensions of $\C$ stable under the action of $G$ correspond to a $G$-invariant version of the data from Proposition~\ref{prop parameterization of subextensions}.
Namely, the subcategory $\tilde{\B} \subset \B$ must be $G$-stable and 
the subgroup $S\subset G(\tilde{\B})$ must be normal in this case.
Since the cohomology class of the obstruction $O_{\tilde{\B}}$ is $G$-invariant,
subextensions will be parameterized by a torsor over $H^1(S,\, N(\tilde{\B}))^G$.
\end{remark}

\subsection{Tannakian subcategories of a gauging}

Let $\B$ be a braided fusion category, let $G$ be a finite group along with a monoidal $2$-functor $F: G\to\uuPic(\B)$,  let
\begin{equation}
\label{extension C}
\C =\bigoplus_{g\in G}\, \C_g,\qquad \C_1=\B,
\end{equation}
be the corresponding central extension,
and let $\Gaug(G,\B,F)=\C^G$ be the corresponding gauging. The goal of this Section is to describe the set $\MT_{max}(\Gaug(G,\B,F))$ of maximal Tannakian subcategories of $\Gaug(G,\B,F)$.

\begin{definition}
Let a group $G$ act on a fusion category $\A$. A braiding $b$ on $\A$ is said to be {\em $G$-invariant} if  $G$ 
acts by braided (with respect to $b$) autoequivalences of $\A$. 
\end{definition}

\begin{proposition}
\label{Tannakian datum}
$T_{max}(\Gaug(G,\B,F))$ is in bijection with the set of triples $(N,\,\E,\, b)$, where $N$ is a normal Abelian subgroup of $G$, $\E =\bigoplus_{n\in N}\,\E_n$ is a $G$-stable subextension of $\C$ such that
$\E_1\in T_{max}^G(\B)$, and $b$ is a $G$-invariant Tannakian braiding on $\E$. 
\end{proposition}
\begin{proof}
Given a triple $(N,\,\E,\, b)$ there is a trivialization of the action of $N$ on $\E$ coming from
the composition of natural isomorphisms
\[
X_n \ot Y \xrightarrow{c_{X_n,Y}} T_n(Y) \ot X_n \xrightarrow{b_{T_n(Y), X_n}}  X_n \ot T_n(Y),\,\qquad Y\in \E,\, X_n \in \E_n,
\]
where $c$ denotes the $G$-crossed braiding of $\C$. This trivialization is  $G$-invariant thanks to the $G$-invariance  of the braiding $b$. 
It follows from Theorem~\ref{GJ theorem} and Remark~\ref{GJ remark} that there is an action of $G/N$ on $\E$ by braided autoequivalences and, consequently,  a  Tannakian subcategory  
\[
\mathcal{T}(N,\,\E,\, b):= \E^{G/N} \subset \bigoplus_{n\in N}\, \C_n.
\]
Note that $\E_1^G$ is a maximal Tannakian subcategory of $\Gaug(G,\B,F)$. We compute: 
\[
\FPdim(\mathcal{T}(N,\,\E,\, b)) = [G:N] \FPdim(\E) = |G| \FPdim(\E_1) = \FPdim(\E_1^G).
\]
Since all maximal Tannakian subcategories of $\C$ have the same Frobenius-Perron dimension, it follows that  $\mathcal{T}(N,\,\E,\, b) \in T_{max}(\Gaug(G,\B,F))$. 

Conversely, suppose that $\mathcal{T}$ is a maximal Tannakian subcategory of $\Gaug(G,\B,F)$. Let $N_\mathcal{T}\subset G$ be a normal subgroup
such that $\mathcal{T} \cap \Rep(G)= \Rep(G/N_\mathcal{T})$ and let 
\[
\E_\mathcal{T} := \mathcal{T} \bt_{\Rep(G/N_\mathcal{T})} \Vect \subset \C
\]
be the corresponding de-equivariantization. The Tannakian 
category $\E_\mathcal{T}$ is an $N_\mathcal{T}$-graded extension of 
$\mathcal{T} \cap \Rep(G)$  and so the action of $N_\mathcal{T}$ on it
is trivializable.   
Hence, the action of $G$ on $\E_\mathcal{T}$ factors through the braided action of 
$G/N_\mathcal{T}$. In particular, $\E_{\mathcal{T}}$ is
$G$-stable. Thus, we get a triple $(N_\mathcal{T},\,\E_\mathcal{T},\, b_\mathcal{T})$ with required properties,
where $b_\mathcal{T}$ denotes the braiding of $\E_\mathcal{T}$. 

Since equivariantization and de-equivariantization  constructions are inverses of each other, the same is true for the assignments
$(N,\,\E,\, b) \mapsto \mathcal{T}(N,\,\E,\, b)$ and $\mathcal{T}\mapsto (N_\mathcal{T},\,\E_\mathcal{T},\, b_\mathcal{T})$.
\end{proof}

\begin{definition}
\label{def imprimitivity}
We will call a triple  $(N,\, \E,\, b)$ with $N\neq\{1\}$ from Proposition~\ref{Tannakian datum}
a {\em Tannakian imprimitivity  datum} for the central extension  \eqref{extension C}
(or for  the corresponding monoidal $2$-functor $F: G \to \uuPic(\B)$ defining it). We say that an
extension (or $2$-functor) is {\em Tannakian primitive} if it does not admit a Tannakian imprimitivity datum.
\end{definition}

\begin{remark}
\label{imprimitivity facts}
\begin{enumerate}
\item[(1)]   
If $F$ admits a non-trivial Tannakian imprimitivity datum $(N,\, \E,\, b)$ then $F|_N$ is imprimitive in  the sense of Definition~\ref{def primitive functor}.  In this case, the central extension corresponding to $F|_N$ is induced from a $G$-stable  $N$-graded Tannakian subextension.
\item[(2)]  
We have $\T(N,\, \E,\, b) \cap \Rep(G)=\Rep(G/N)$. The support of the  image of $\T(N,\,\E,\,b)$ under the de-equivariantization  functor  $\C^G\to \C$ is $N$.  
\item[(3)]     
A parameterization of Tannakian imprimitivity data
for a central extension \eqref{extension C}
can be given using the machinery of \cite{DN2}, where 
the extension theory of symmetric categories was developed. We prefer not to do it here to avoid extra technical details.
\end{enumerate}
\end{remark}

\subsection{Lagrangian subcategories 
of a twisted Drinfeld double}

Here we apply results of the previous Section to describe Lagrangian subcategories of gaugings of the trivial fusion category $\Vect$. By \cite[Section 4.4.10]{DGNO2}, these are precisely representation categories of twisted Drinfeld doubles of finite groups. 

Let $N$ be a normal Abelian subgroup of a finite group $G$.
Let $H^3(G,\,k^\times)_1$ denote the kernel of the restriction homomorphism
$H^3(G,\,k^\times) \to H^3(N,\,k^\times)$.

\begin{remark}
\label{H3 to H1 hom}
There is a canonical homomorphism
\begin{equation}
\label{mw}
m: H^3(G,\,k^\times)_1 \to H^1(G,\, H^2(N,\,k^\times)): \omega \mapsto m_{\omega,N},
\end{equation}
where $H^2(N,\,k^\times)$ is a right $G$-module via $(g\cdot\mu)(x,y)= \mu(gxg^{-1},gyg^{-1}),\,
g\in G,x,y\in N$.
It is defined as follows.
The cochain
\begin{equation}
\label{mug}
\mu_g(y,z) = \frac{\omega(gyg^{-1}, gzg^{-1}, g)\omega(g,y,z)}{\omega(gyg^{-1}, g, z)}
\end{equation}
defines a tensor structure on the conjugation autoequivalence $\ad(g): \oplus V_x \mapsto  \oplus V_{gxg^{-1}} ,\, g\in G$,   
of $\Vect_G^\omega$, that is, satisfies the identity
\begin{equation}
\label{d2mu}
\frac{\mu_g(xy,z)\mu_g(x,y)}{\mu_g(x,yz)\mu_g(y,z)} = \frac{\omega(gxg^{-1}, gyg^{-1}, gzg^{-1})}{\omega(x,y,z)}
\end{equation}
for all $g, x,y,z \in G$.
Let $d$ and $\partial$ denote differentials in $C^\bullet(G,\, k^\times)$ and in $C^\bullet(G,\, H^2(N,\, k^\times))$, respectively.
Let $\nu \in C^2(G,\, k^\times)$ be  such that $d\nu|_N =\omega|_N$.  It follows from \eqref{d2mu} that
\begin{equation}
\label{mwN}
m_{\omega, N}(g) := \mu_g|_N  \times \frac{\nu}{\nu^g}, \qquad  g \in G,
\end{equation}
is a $2$-cocycle in $Z^2(N,\, k^\times)$.  The cochain $m_{\omega,N} \in C^1(G,\, H^2(N,\, k^\times))$ depends on the choice of $\nu$,
but its cohomology class does not. Indeed, if $\nu'$ is another choice, then  $\nu'|_ N = \nu|_N \times \eta$, where  
$\eta\in  H^2(N,\, k^\times) = C^0(G,\, H^2(N,\, k^\times))$ 
and, hence, $m_{\omega, N}$ changes to $m_{\omega, N} \partial(\eta^{-1})$.
To see that $m_{\omega, N}$ is a $1$-cocycle, we compute
\[
\left(\partial(m_{\omega,N})(g,h)\right) (y, z) = \frac{\mu_{gh}(y,z)}{ \mu_{g}(hyh^{-1},hzh^{-1}) \mu_{h}(y,z)} = \frac{\gamma_{g,h}(yz)}{\gamma_{g,h}(y)\gamma_{g,h}(z)},
\qquad g,h,\in G,\, y,z\in N,
\]
where $\gamma_{g,h}(x) = \frac{\omega(g, hxh^{-1}, h)}{\omega(g,h,x) \omega(ghxh^{-1}g^{-1}, g, h)}$, $g,h,\in G,\, x\in N$. Thus, $\partial(m_{\omega,N})(g,h)=0$  and  the cohomology class of $m_{\omega,N}$ is a well defined element of $H^1(G,\, H^2(N,\, k^\times))$. This element depends on $\omega$ multiplicatively, hence 
the map \eqref{mw} is a group homomorphism.
\end{remark}


\begin{remark}
Elements of $H^1(G,\, H^2(N,\,k^\times))$ correspond to isomorphism classes of monoidal functors $G \to \uAut^\ot(\Vect_N)$ that restrict to the adjoint action of
$G$ on $N$.
\end{remark}

Let
\begin{equation}
\label{OmegaGN}
\Omega(G;N):= \text{Ker}\left( H^3(G,\,k^\times)_1 \xrightarrow{m}  H^1(G,\, H^2(N,\,k^\times) ) \right).
\end{equation}

\begin{remark}
The group $\Omega(G;N)$ appeared in \cite{U} and is related to the problem of classification of categorical Morita equivalence classes of pointed fusion categories.
\end{remark}

The following result was established in \cite{NN}. We reformulate it in more transparent terms
and derive it as a consequence of Proposition~\ref{Tannakian datum}.

\begin{proposition}
\label{Lagrangians in ZVecGw}
There is a bijection between $T_{max}(\Z(\Vec_G^\omega))$  
and the set of pairs $(N,\, b)$, where $N$ is a normal Abelian subgroup of $G$ such that 
$\omega \in \Omega(G;N)$  and $b$ is an element of an $H^2(N,\, k^\times)^G$-torsor.
\end{proposition}
\begin{proof}
In the notation of Proposition~\ref{Tannakian datum}, we have $\E_1=\Vect$ and  $\E\cong \Vect_N$, where
$N$ is a normal Abelian subgroup of $G$.  As in Remark~\ref{H3 to H1 hom}, let 
$\nu \in C^2(G,\, k^\times)$ be  such that $d\nu|_N =\omega|_N$ and let $\mu_g,\, g\in G,$ be defined by equation \eqref{mug}. The Tannakian braiding on $\E$ is given by  a function  $b: N \times N\to k^\times$ such that  $B:=b\, \Alt(\nu)$, where $\Alt(\nu)(x,y)=\frac{\nu(x,y)}{\nu(y,x)},\, x,y\in N,$ is an alternating  bilinear form on $N$.  The $G$-invariance of this braiding is equivalent to the condition
\begin{equation}
\label{G-invariance of b}
\frac{b^g}{b}\, \Alt(\mu_g) =1,
\end{equation}
where $b^g(x,y)= b(gxg^{-1},\, gyg^{-1}),\, g\in G,\,x,y\in N$. Equivalently,
\begin{equation}
\label{BgB}
\frac{B^g}{B}\, \Alt(m_{\omega,N}(g)) =1,\qquad g\in G,
\end{equation}
where $m_{\omega,N}$ was defined in \eqref{mwN}. This translates to $\Alt(m_{\omega,N}(g))$
being a coboundary in $C^1(G,\, H^2(N,\,k^\times))$, i.e., 
to  $m_{\omega,N}(g)$ being  cohomologically trivial. Different choices  $b,\, b'$ of the braiding
on $\E$ give an alternating form $B$ satisfying \eqref{BgB} if and only if $\frac{b'}{b}$ is a $G$-invariant
alternating bilinear form on $N$. Since $H^2(N,\, k^\times)^G \cong \Hom(\bigwedge^2 N,\, k^\times)^G$,
we see that $b$ is determined up to an element of an $H^2(N, k^\times)^G$-torsor.
\end{proof}

\begin{remark}
Proposition~\ref{Lagrangians in ZVecGw} can be interpreted as follows. Given a normal Abelian
subgroup $N \subset G$ there is an obstruction $m_{\omega,N}\in H^1(G,\,H^2(N,\,k^\times))$ to the existence of Lagrangian subcategories of $\Z(\Vect_G^\omega)$ supported on $N$. When this obstruction vanishes,
the set of Lagrangian subcategories supported on $N$ is parameterized by a torsor over
$H^0(G,\,H^2(N,\,k^\times)) = H^2(N, k^\times)^G$.
\end{remark}

\begin{remark}
 When $\omega =1$, the $H^2(N, k^\times)^G$-torsor is trivial and $L(\Z(\Vect_G))$ is in bijection with
 pairs $(N,\, B)$, where $B$ is an alternating $G$-invariant bilinear form on $N$ \cite{NN}.
\end{remark}


\section{Tannakian radical, mantle, and reductive braided fusion categories}
\label{section Rad(B)}

\subsection{Radical and mantle}
\label{sect: radical in mantle}

Let $\B$ be a braided fusion category. 

\begin{definition}
\label{def Tannakian radical}
The {\em Tannakian radical} of $\B$ is the intersection of all maximal Tannakian subcategories of $\B$.
\end{definition}

Let $\Rad(\B)$ denote the Tannakian radical of $\B$. This is a Tannakian category, so by Deligne's theorem
there is canonical equivalence 
\begin{equation}
\label{eq RadB=Rep  GB}
\Rad(\B) \cong\Rep(G_{\B})
\end{equation}
for a unique (up to isomorphism) finite group $G_{\B}$.
 
\begin{definition}
\label{def radical group}
The group $G_\B$ from \eqref{eq RadB=Rep  GB} will be called {\em the radical group} of $\B$.
\end{definition}

 \begin{definition}
\label{def mantle}
The {\em mantle}\footnote{The term `mantle' is  used in geology to describe the layer surrounding the Earth's core. It extends outward and encompasses more material. This is similar to the relation between $\Mantle(\B)$ and $\Core(\B)$
in our setting since the latter is a localization of the former.  We hope
this analogy justifies our choice of the name.}
of $\B$ is the localization of $\B$ with respect to $\Rad(\B)$, i.e.,
\begin{equation}
\label{eq Mantle deeq}
\Mantle(\B) = \Rad(\B)' \bt_{\Rad(\B)} \Vect.
\end{equation}
\end{definition}

Clearly, $G_\B$ and $\Mantle(\B)$ are invariants of $\B$. Furthermore, there is a canonical 
normal subgroup $H_\B \subset G_\B$ such that  $\B \boxtimes_{\Rad(\B)} \Vect$ is a faithfully $H_\B$-graded extension of $\Mantle(\B)$ \cite[Section 4.4.8]{DGNO2}. Equivalently, there is  a canonical monoidal $2$-functor
\begin{equation}
\label{eq Mantle functor}
F_\B: H_\B \to \uuPic(\Mantle(\B)).
\end{equation}
Recall that a braided fusion category is {\em almost non-degenerate} if it is either non-degenerate or slightly degenerate.
For an almost non-degenerate category $\B$ we have $H_\B=G_\B$ and 
$\B$ is equivalent to a $G_\B$-gauging of $\Mantle(\B)$.
In this case,
the triple $(G_\B,\,\Mantle(\B),\,F_\B)$ is a complete invariant of $\B$. 

\begin{proposition}
\label{FB is primitive}
Let $\B$ be an almost non-degenerate braided fusion category. The canonical 
monoidal $2$-functor $F_\B:G_\B \to \uuPic(\Mantle(\B))$ is Tannakian primitive.
\end{proposition}
\begin{proof}
Proposition~\ref{Tannakian datum} gives a parametrization  of $\MT_{max}(\B)$ in terms of imprimitivity data $(N,\, \E,\, b)$ associated to $F_\B$, see Definition~\ref{def imprimitivity}. Let $\T(N,\, \E,\,b)$ denote the
corresponding maximal Tannakian subcategory of $\B$.
By Remark~\ref{imprimitivity facts}(2), 
\[
\Rep(G_\B) \cap \T(N,\, \E,\,b) = \Rep(G_\B/N). 
\]
Since $\Rad(\B)= \Rep(G_\B)$ is contained in every maximal Tannakian subcategory of $\B$, we conclude that $N=\{1\}$,
so there is no non-trivial Tannakian imprimitivity data for $F_\B$.    
\end{proof}


Let $G$ be a finite group and let  $N(G)$ be the subgroup of $G$ generated by all normal Abelian subgroups of $G$.
For $\omega\in H^3(G,\, k^\times)$, let $N^\omega(G)$ be the subgroup  of $N(G)$ generated by normal Abelian subgroups $N \subset G$
such that $\omega \in \Omega(G;N)$, see \eqref{OmegaGN}.
The groups $N(G)$ and $N^\omega(G)$ are nilpotent normal subgroups of $G$, as a product of nilpotent normal subgroups of a finite group is nilpotent.

\begin{proposition}
\label{ex RadZVectG}
We have $\Rad(\Z(\Vect_G^\omega)) \cong  \Rep(G/N^\omega(G))$  and 
\begin{equation} 
\Mantle(\Z(\Vect_G^\omega)) \cong \Z(\Vect_{N^\omega(G)}^{\omega|_{N^\omega(G)}}).
\end{equation}
\end{proposition}
\begin{proof}
By Proposition~\ref{Lagrangians in ZVecGw}, Lagrangian  subcategories of $\Z(\Vect_G^\omega)$ are parameterized by 
pairs $(N,\, B)$, where $N$ is a normal Abelian subgroup of $G$ and $B$ is an element of a (possibly empty) $H^2(N,\,k^\times)^G$-torsor.
By Remark~\ref{imprimitivity facts}(2),  the intersection of the Lagrangian category corresponding to $(N,\,B)$  with $\Rep(G)$ is $\Rep(G/N)$. Therefore,
\[
\Rad(\Z(\Vect_G^\omega)) =\bigcap_{\substack{\text{$N$ is normal Abelian,} \\ {\omega \in \Omega(G;N)}}}\, \Rep(G/N) = \Rep(G/N^\omega(G)).
\]
The centralizer $\Rep(G/N^\omega(G))'$  consists of objects of $\Z(\Vect_G^\omega)$ supported on $N^\omega(G)$ which is precisely
$(\Vect_{N^\omega(G)}^{\omega|_{N^\omega(G)}})^G$. The mantle of $\Z(\Vect_G^\omega)$ is the $G/N^{\omega}(G)$-de-equivariantization of the latter,
which is $\Z(\Vect_{N^\omega(G)}^{\omega|_{N^\omega(G)}})$
\end{proof}

\subsection{Reductive categories and their properties}
\label{sect: reductive}

\begin{definition}
\label{def reductive}
A braided fusion category $\B$ is {\em reductive} if $\Rad(\B)= \Vect$.
\end{definition}

\begin{remark}
\label{reductive imply almost degenerate}
Note that the maximal Tannakian subcategory of $\Z_{sym}(\B)$ is contained in $\Rad(\B)$. Therefore, if $\B$ is reductive, then  $\Z_{sym}(\B) = \text{$\Vect$ or $\sVect$}$, i.e., $\B$ is almost non-degenerate.
\end{remark}

\begin{remark}
\label{rem nilpotency}
It follows from Proposition~\ref{ex RadZVectG} that  $\Z(\Vect_G^\omega)$ is reductive if and only if $G=N^\omega(G)$. The latter condition implies that $G$
is nilpotent. The converse is false in general.
For example, if $G$ is the dihedral group of order $2^n,\, n\geq 4$, then $N(G) \subsetneq G$.
\end{remark}

\begin{example}
Here are some examples of groups $G$ such that $N(G)=G$.
\begin{enumerate}
    \item[(1)] Let $G$ be a nilpotent group of nilpotency class $\leq 2$, i.e., such that $G/Z(G)$ is Abelian, where $Z(G)$ denotes the center of $G$.  Any cyclic subgroup of $G/Z(G)$ gives rise to a normal Abelian subgroup of $G$. These subgroups generate $G$, so $N(G)=G$ in this case.
    \item[(2)] Let $G=UT(n,\, F)$ be the group of upper unitriangular $n\times n$ matrices, $n\geq 2$, with entries from a finite field $F$. This is a nilpotent group of class $n-1$ generated by matrices $e_{ij}(\lambda),\, 1\leq i< j \leq n,\, \lambda \in F,$ whose only nonzero entry off the main diagonal is $\lambda$ in the $(i,j)$-position.  The group $G$
    is generated by normal Abelian subgroups $N_k =\langle e_{ij}(\lambda) \mid 1\leq k < j,\, \lambda \in F \rangle$, $k=1,\dots, n$.
    \end{enumerate}
\end{example}

\begin{proposition}
\label{maxTan of red nil}
A maximal Tannakian subcategory of a reductive braided fusion category is nilpotent.    
\end{proposition}
\begin{proof}
Let $\B$ be a reductive braided fusion category and let $\E=\Rep(G)$ be its maximal Tannakian subcategory.
Then $\B \cong \C^G$, where $\C$ is a $G$-crossed braided fusion category.  For
any Tannakian subcategory $\F \subset \B$ let $N_\F \subset G$ be a normal subgroup such that 
$\E \cap \F =\Rep(G/N_\F) \subset \Rep(G)$.  By \cite[Proposition 4.56]{DGNO2}, the support in $G$ 
of the image of $\F$ under the de-equivariantization functor $\C^G\to \C$ is $N_\F$. Hence, $N_\F$
is Abelian. Since $\bigcap_{\F \in T_{max}(\B)}\,\Rep(G/N_\F)= \Rad(\B) =\Vect$,
normal Abelian subgroups $\N_\F$ generate $G$, therefore, $G$ is nilpotent.
\end{proof}

\begin{proposition}
\label{red WGT nil}  
A reductive braided fusion category is weakly group-theoretical if and only if
it is nilpotent.
\end{proposition}
\begin{proof}
By definition, a nilpotent fusion category is weakly group-theoretical.
Conversely, let $\B$ be a reductive weakly group-theoretical braided fusion category and let $\E=\Rep(G)$ be a maximal Tannakian subcategory of $\B$.
By Proposition~\ref{maxTan of red nil}, $G$ is nilpotent.
Furthermore, the core of $\B$ is a nilpotent braided fusion category \cite{Na}. 
Therefore, $\C= \B\bt_{\E} \Vect$ is nilpotent, since it is a graded extension
of the core. Hence, $\Z(\C)$ is nilpotent \cite{GN}.
Since $\B=\C^G$ is equivalent to a subcategory of $\Z(\C)$ \cite[Remark 4.43(i)]{DGNO2},  it follows from \cite{GN} that $\B$ is nilpotent.   
\end{proof}

\begin{proposition}
\label{reductive equivariantization}
Let $\B$ be a braided fusion category and let $G$ be a group acting on $\B$
by braided autoequivalences. Then $\Rad(\B) \subset \Rad(\B^G)\bt_{\Rep(G)} \Vect$.
\end{proposition}
\begin{proof}
Any maximal Tannakian subcategory $\mathcal{T} \subset \B^G$ contains $\Rep(G)$ as a subcategory
and so  $\mathcal{T} =\E^G$, where $\E$ is a maximal $G$-stable Tannakian subcategory of $\B$. Hence, $\Rad(\B)^G \subset \mathcal{T}$. Therefore, 
\[
\Rad(\B)^G \subset \bigcap_{\E\in T_{max}^G(\B)} \E^G = \Rad(\B^G).
\]
The result follows by taking the de-equivariantization of both sides.
%
\end{proof}

\begin{definition}
Let $\B$ be a braided fusion category and let $G$ be a group acting on $\B$
by braided autoequivalences.  Define a {\em $G$-radical} of $\B$ by
\[
\Rad^G(\B) :=\bigcap_{\E \in \MT_{max}^G(\B)}\E.
\]
If $\Rad^G(\B)=\Vect$ we will say that $\B$ is {\em $G$-reductive}.
\end{definition}

\begin{corollary}
We have $\Rad^G(\B)=\Rad(\B^G) \bt_{\Rep(G)} \Vect$
and $\Rad(\B)\subset \Rad^G(\B)$. In particular, 
if $\B$ is $G$-reductive, then it is reductive.
\end{corollary}
\begin{proof}
This follows from the proof of Proposition~\ref{reductive equivariantization} since
(de-)equivariantization commutes with taking the intersection of subcategories. 
\end{proof}

\begin{corollary}    
\label{Mantle reductive}
Let $\B$ be a braided fusion category. Then $\Mantle(\B)$ is $G_\B$-reductive.
In particular, $\Mantle(\B)$ is reductive. 
\end{corollary}
\begin{proof}
By definition of the radical, any maximal Tannakian subcategory of $\B$
is contained in $\Rad(\B)' =\Mantle(\B)^{G_\B}$ and, hence, 
\[
\Rad(\Mantle(\B)^{G_\B})  =\Rad(\B) =\Rep(G_\B).
\]
Combining this with Corollary~\ref{Mantle reductive}, we have
\[
\Rad^{G_\B}(\Mantle(\B)) = \Rad(\Mantle(\B)^{G_\B}) \bt_{\Rep(G_\B)} \Vect \cong \Vect,
\]
as required.
\end{proof}

\subsection{A complete invariant of a braided fusion category}
\label{Sect complete invrariant}

\begin{definition}
\label{def CGT}
A  triple $(G,\, \C,\, F)$, 
where $G$ is a finite group,  $\C$ is a  braided fusion category,
and $F: G \to \uuPic(\C)$ is a monoidal $2$-functor such that
\begin{enumerate}
    \item[(1)] $F$ is Tannakian primitive in the sense of Definition~\ref{def imprimitivity} and
    \item[(2)] $\C$ is $G$-reductive with respect to the action of $G$ on $\C$ coming from $F$.
\end{enumerate}
will be called a {\em canonical gauging triple}.  
\end{definition}

A $1$-isomorphism between canonical gauging triples $(G,\, \C,\, F)$ and $(\tilde{G},\, \tilde{\C},\, \tilde{F})$ is a triple
$(i,\, I,\,\iota)$
consisting of a group isomorphism $i: G \to \tilde{G}$, a braided equivalence $I: \tilde{\C} \to \C$, and a
pseudo-natural isomorphism of monoidal $2$-functors 
\begin{equation}
\label{iota}
\xymatrix{
G \ar[d]_{i}_{}="a" \ar[rr]^{F} && \uuPic(\C) \ar[d]^{I^*}="b" \\
\tilde{G} \ar[rr]^{\tilde{F}} && \uuPic(\tilde{\C}),
\ar@{}"a";"b"^(.35){}="b2"^(.65){}="c2" \ar@{=>}^{\iota}"b2";"c2"
}
\end{equation}
where $I^*: \uuPic(\C) \to \uuPic(\tilde{\C})$ is a monoidal $2$-equivalence induced from $I$.

If $(i',\, I',\,\iota')$ is another isomorphism between 
$(G,\, \C,\, F)$ and $(\tilde{G},\, \tilde{\C},\, \tilde{F})$
then a $2$-isomorphism between $(i,\, I,\,\iota)$ and $(i',\, I',\,\iota')$ is a pair $(g,\,\Gamma)$, where $g\in G$ is such that $i' = i\circ \ad(g)$ and  $\Gamma: I'  \xrightarrow{\sim} 
g \circ I$ is a natural isomorphism of monoidal functors (here $g\in G$ is viewed as the autoequivalence of $\C$
obtained via the composition $G\to \uuPic(\C) \to \uAut^{br}(\C)$)
such that the following prism of $2$-cells commutes:
\begin{equation}
\label{prism}
\xymatrix{
&&& G \ar[ddll]_{F} \ar@{-->}[dd]^{\ad(g)}="a" \ar[ddrrrr]^{i'}="c" &&  &&  \\
&&& && && \\
& \uuPic(\C)  \ar[dd]_{g^*}="b"  \ar[ddrrrr]^{{I'}^*}="d" &&  G
\ar@{-->}[ddll]_{F} \ar@{-->}[rrrr]_{i}="e"  && && \tilde{G}  \ar[ddll]^{\tilde{F}}  \\
 &&& && && \\
& \uuPic(\C)  \ar[rrrr]_{I^*}="f" && && \uuPic(\tilde{\C}), &&
\ar@{}"a";"b"^(.30){}="x"^(.60){}="y" \ar@2{-->}^{\text{can}}"x";"y"
\ar@{}"c";"d"^(.25){}="x"^(.75){}="y" \ar@2{->}_{\iota}"x";"y"
\ar@{}"c";"e"^(.25){}="x"^(.75){}="y" \ar@2{=}^{}"x";"y"
\ar@{}"d";"f"^(.25){}="x"^(.75){}="y" \ar@2{->}^{\Gamma^*}"y";"x"
\ar@{}"e";"f"^(.10){}="x"^(.60){}="y" \ar@2{-->}_{\iota'}"x";"y"
}
\end{equation}
where $\Gamma^*$ denotes the a pseudo-natural isomorphism of monoidal $2$-functors induced by $\Gamma$ and $\text{can}$
is the canonical pseudo-natural isomorphism from \cite[Theorem 3.10]{DN1}.

It follows that canonical gauging triples form a $2$-groupoid
which we will denote $\mathfrak{G}$.

Given a canonical gauging triple $(G,\C,F)$, the corresponding gauging $\Gaug(G,\C,F)$ is a braided fusion category. The latter category is non-degenerate (respectively, slightly degenerate)
if and only if $\C$ is non-degenerate (respectively, slightly degenerate).

Conversely, in Section~\ref{sect: radical in mantle}  to an almost non-degenerate  braided fusion category $\B$ we associated a  canonical gauging triple
$\mathcal{T}(\B):=(G_\B,\, \Mantle(\B),\, F_\B)$.

\begin{proposition}
\label{functor conditions}
Consider a triple $(G,\, \C,\, F)$, where $G$ is a finite group, $\C$ is a reductive braided fusion category,
and $F: G \to \uuPic(\C)$ is a monoidal $2$-functor. We have   
$\Rad(\Gaug(G,\C,F)) = \Rep(G)$ if and only if the following two conditions are satisfied:
\begin{enumerate}
    \item[(1)] The functor $F$ is Tannakian primitive in the sense of Definition~\ref{def imprimitivity},
    \item[(2)] $\Rad^G(\C)=\Vect$.
\end{enumerate}
\end{proposition}
\begin{proof}
Let $\B = \Gaug(G,\C,F)$. 
The first condition is equivalent to  $\Rep(G)$ being contained in every $\E\in \MT_{max}(\B)$, while the second is equivalent to $\Rep(G)$ not being a proper subcategory of $\Rad(\B)$.
\end{proof}

Let $\mathfrak{B}$ denote the $2$-groupoid of almost non-degenerate
braided fusion categories with  braided equivalences as $1$-morphisms and natural isomorphisms of tensor functors as $2$-morphisms.

\begin{theorem}
\label{CGT is a complete invariant} 
The assignments
\begin{equation}
\mathcal{T}: \mathfrak{B}\to \mathfrak{G}: \B \mapsto (G_\B,\, \Mantle(\B),\, F_\B)
\qquad \text{and} \qquad
\Gaug: \mathfrak{G}\to \mathfrak{B} : (G,\, \C,\, F) \mapsto \Gaug(G,\C,F)
\end{equation}
are mutually quasi-inverse $2$-equivalences of $2$-groupoids.
\end{theorem}
\begin{proof}
It follows from Propositions~\ref{FB is primitive}
and \ref{functor conditions} that $\mathcal{T}\Gaug(G,\C,F) \cong (G,\C,F)$.  We have $\Gaug\mathcal{T}(\B) \cong \B$ since the equivariantization and de-equivariantization $2$-functors are inverses of each other \cite[Theorem 4.44]{DGNO2}.

An equivalence $\B_1\xrightarrow{\sim} \B_2$ between braided fusion categories restricts to a braided equivalence 
$\Rad(\B_1) \xrightarrow{\sim} \Rad(\B_2)$ and, hence, 
gives an equivalence of crossed braided extensions 
\begin{equation}
\label{crossed equivalence}
\B_1 \bt_{\Rad(\B_1)}\Vect \cong 
\B_2 \bt_{\Rad(\B_2)}\Vect.
\end{equation}
In particular, there is a braided equivalence $I:\Mantle(\B_1) \xrightarrow{\sim} \Mantle(\B_2)$. The equivalence  \eqref{crossed equivalence} yields an isomorphism of the corresponding monoidal $2$-functors \cite{ENO2}. Thus, a $1$-morphism in $\mathfrak{B}$
yields a $1$-morphism in $\mathfrak{G}$. Conversely, a $1$-morphism
between gauging triples gives rise to a braided equivalence between the corresponding gaugings. These assignments are inverse
to each other, again thanks to the equivariantization - 
de-equivariantization correspondence.

Finally, a natural isomorphism $\eta$ between braided autoequivalences $\alpha,\alpha': \B_1 \to \B_2$ yields an isomorphism  $\eta|_{\Rad(\B_1)}$ between their restrictions $\alpha|_{\Rad(\B_1)},
\alpha'|_{\Rad(\B_1)}: \Rad(\B_1) \to \Rad(\B_2)$. 
Let $\iota,\iota': G_{\B_1} \to G_{\B_2}$ denote the corresponding
isomorphisms between the radical groups. Then $\eta|_{\Rad(\B_1)}$
is given by 
$\eta|_V = \iota(g)|_V,\, V \in \Rad(\B_2)= \Rep(G_{\B_2})$ for a unique $g\in G_{\B_1}$ such that $\iota' =\iota \circ \ad(g)$.
Furthermore, a translation by $g$ is an automorphism of the regular algebra $A_{\B_1} =\Fun(G_{\B_1},\, k)\in \Rep(G_{\B_1})$. By definition, $\Mantle(\B_1)$ is the category of $A_{\B_1}$-modules in $\Rad(\B_1)'$, the centralizer of $\Rad(\B_1)$ in $\B_1$, and the above action of $g$ on $A_{\B_1}$ becomes a braided autoequivalence of $\Mantle(\B_1)$.
Thus, $\eta|_{\Rad(\B_1)'}$ yields a natural isomorphism 
$\Gamma: I'\xrightarrow{\sim} g \circ I$, where
 $I , I': \Mantle(\B_2)
\to \Mantle(\B_1)$ are the braided equivalences coming from $\alpha,\alpha'$. Furthermore, there are induced
equivalences of $G$-crossed braided categories
\begin{equation}
\label{bar alpha}
\bar\alpha,\, \bar\alpha' : \B_1\bt_{\Rad(\B_1)} \Vect \xrightarrow{\sim} 
\B_2\bt_{\Rad(\B_2)} \Vect
\end{equation}
and $\eta$ induces a natural isomorphism $\bar\eta$ between these.
The fact that an $\bar\eta$ is an extension of $\Gamma$ is equivalent to the commutativity of the polytope \eqref{prism}.

Thus, $\mathcal{T}$ and $\Gaug$ are mutually quasi-inverse $2$-equivalences.
\end{proof}

\subsection{Towards a parametrization of Morita equivalence classes of fusion categories}

It was shown in \cite{ENO2} that fusion categories $\A_1,\, \A_2$ are categorically 
Morita equivalent if and only if their centers $\Z(\A_1)$ and $\Z(\A_2)$ are braided equivalent.
In other words, the center is a complete invariant of the Morita equivalence class of a fusion category.

\begin{corollary}
\label{complete Morita invariant}
The canonical gauging triple $(G_{\Z(\A)},\, \Mantle(\Z(\A)),\,
F_{\Z(\A)})$ is a complete invariant of the Morita equivalence class of a fusion category $\A$. In particular, the radical and mantle of $\Z(\A)$ are Morita invariants of $\A$.    
\end{corollary}

The mantle of $\Z(\A)$ is a reductive metabolic braided fusion category, see  Definition~\ref{def metabolic}. In particular, it is braided equivalent to the representation category of a twisted Drinfeld double. This double is not uniquely defined. A classification of  braided equivalence classes of Drinfeld doubles is equivalent to a classification of pointed fusion categories up to categorical Morita equivalence. The corresponding symmetric relation on the set of pairs $(G,\omega),\, \omega \in H^3(G,\,k^\times)$, was described in \cite{N, U}. Our theory reduces the parametrization of these classes to the problem of classifying reductive metabolic braided fusion categories and computing their categorical Picard groups. As we saw in Remark~\ref{rem nilpotency} and Proposition~\ref{red WGT nil}, such categories (and underlying groups $G$) are nilpotent. We discuss pointed reductive categories in Section~\ref{Sect reductive pointed} and classify
the smallest non-pointed metabolic reductive categories in Proposition~\ref{Prop p+1 classes}.


\section{The structure of reductive braided fusion categories}
\label{Sect  structure reductive}

\subsection{Tannakian-generated categories}

\begin{definition}
A braided fusion category $\B$ will be called a {\em Tannakian-generated category} or {\em TG-category} if it is generated
by its Tannakian subcategories, i.e., $\B =\bigvee_{\E \in \MT(\B)}\, \E$.
\end{definition}

\begin{proposition}
\label{saturated WGT}
A TG-category is group-theoretical.
\end{proposition}
\begin{proof}
Let $\D,\, \E$ be group-theoretical fusion subcategories of a  TG-category $\B$. The tensor product of $\B$ determines
a surjective tensor functor $\D \bt \E \to \D \vee \E$ and the latter is group-theoretical by \cite{ENO1}.
Hence, the induction on the number of Tannakian subcategories needed to generate $\B$ shows that $\B$ is group-theoretical. 
\end{proof}

\begin{proposition}
\label{TG properties}
\begin{enumerate}
\item[(i)] An almost non-degenerate TG-category is reductive.
\item[(ii)] A metabolic braided fusion category is a TG-category if and only if it is reductive.
\item[(iii)] A reductive $TG$-category is integral and nilpotent.
\end{enumerate}
\end{proposition}
\begin{proof}
(i) If a TG-category $\B$ is almost non-degenerate then
\[
\Rad(\B) =  \bigcap_{\E \in \MT_{max}(\B)}\, \E \subset  \left(  \bigvee_{\E \in \MT_{max}(\B)}\, \E' \right)'
\subset  \left( \bigvee_{\E \in \MT_{max}(\B)}\, \E \right)' = \Z_{sym}(\B),
\]
so $\Rad(\B)$ must be trivial.

(ii)  For a metabolic category $\B$ its subcategories $\bigcap_{\E \in \MT_{max}(\B)}\, \E$ and $\bigvee_{\E \in \MT_{max}(\B)}\, \E$
are centralizers of each other. So the former is trivial if and only if the latter coincides with $\B$. 

(iii) This follows from Propositions~\ref{red WGT nil} and \ref{saturated WGT}.
\end{proof}

\begin{remark}
\label{Sylow remark}
It follows from Proposition~\ref{TG properties}(iii) and  \cite[Theorem 6.10]{DGNO1} that a reductive TG-category
is a tensor product of TG-categories of prime power dimension.
\end{remark}

\begin{definition}
\label{def coradical}
The {\em coradical} of a braided fusion category $\B$ 
is   the maximal TG-subcategory of $\B$:
\begin{equation}
\label{Bsat}
\Corad(\B): = \bigvee_{\E \in \MT(\B)}\, \E. 
\end{equation}
\end{definition}

\begin{remark}
\label{basic corad}   
We have $\Rad(\B)= \Rad(\Corad(\B))$.  In particular,
$\B$ is reductive if and only if $\Corad(\B)$ is reductive. 
\end{remark}

\begin{proposition}
For any braided fusion category $\B$ we have $\Corad(\B)\subset \Rad(\B)'$.
When $\B$ is metabolic, $\Corad(\B) = \Rad(\B)'$.
\end{proposition}
\begin{proof}
Since $\E \subset \E'$ for any Tannakian subcategory $\E\subset \B$, we have
\[
\Corad(\B)=\bigvee_{\E \in \MT_{max}(\B)}\E \subset \bigvee_{\E \in \MT_{max}(\B)}\E'
\subset \left(\bigcap_{\E \in \MT_{max}(\B)} \E\right)' =\Rad(\B).'
\]
When $\B$ is metabolic, $\E = \E'$ for every $\E \in T_{max}(\B)$ and the above inclusions become equalities.
\end{proof}

\begin{corollary}
\label{saturated part of primitive}
Let $\B$ be a reductive braided fusion category. Then $\Corad(\B)'$  is anisotropic, i.e., has 
no non-trivial Tannakian subcategories. In particular, $\Corad(\B)'$  is almost non-degenerate.
\end{corollary}
\begin{proof}
By Remark~\ref{basic corad}, $\Corad(\B)$ is reductive and, hence, is almost non-degenerate. 
Therefore, for any Tannakian subcategory $\E \subset \Corad(\B)'$ we have
\[
\E \cap \Corad(\B)= \E \cap \Z_{sym}(\Corad(\B)) =\Vect.
\]
Since $\Corad(\B)$ contains every Tannakian subcategory of $\B$, we conclude that $\E$ must be trivial.
\end{proof}

Thus, a reductive braided fusion category $\B$ has a canonical TG-subcategory $\Corad(\B)$
and a canonical anisotropic subcategory $\Corad(\B)'$.

\subsection{Projectively isotropic objects}

Let $\B$ be a braided fusion category.

\begin{definition}
\label{types of centralizing}  
We will say that an object $X\in \B$ is {\em symmetric} if $c_{X,X}^2=\id_{X\ot X}$
and {\em projectively symmetric} if  $c_{X,X}^2=\lambda\, \id_{X\ot X}$  for some  $\lambda\in k^\times$.
\end{definition}

For any object $X$ of $\B$ let $\langle X \rangle$ denote the fusion subcategory of $\B$ generated by $X$.
Clearly, $X$ is symmetric if and only if $\langle X \rangle$ is symmetric.

\begin{definition}
\label{def isotropic}  
A symmetric object $X$ is {\em isotropic} if $\langle X \rangle$ is Tannakian.
\end{definition}

If $X$ is projectively symmetric then $\langle X\ot X^* \rangle$ is symmetric by \cite[Proposition 3.22]{DGNO2}.

\begin{definition}
\label{def proj isotropic}  
A projectively symmetric object $X$ is {\em projectively isotropic} if $\langle X\ot X^* \rangle$ is Tannakian.
\end{definition}

\begin{remark}
Our terminology is consistent with the one used in \cite[Section 3.3]{DGNO2}.  
\end{remark}


\begin{lemma}
\label{proj iso is int}
Let $X$ be a projectively isotropic object of $\B$. Then $\FPdim(X)\in \mathbb{Z}$.   
\end{lemma}
\begin{proof}
It follows from \cite[Lemma 3.15]{DGNO2}  that $X$ is a homogeneous object
in a graded extension of some Tannakian subcategory $\E=\Rep(G) \subset \B$. But any homogeneous component of this extension is equivalent to the category of projective representations of $G$ with a fixed Schur multiplier \cite{DN2}. Therefore, $\FPdim(X)$ is equal to the degree of a projective representation of $G$.
\end{proof}

Let $\B_{int}$ be the maximal integral fusion subcategory of $\B$, i.e., the subcategory generated by 
all objects $X\in \B$ such that $\FPdim(X)\in \mathbb{Z}$. By \cite[Corollary 2.24]{DGNO2}, there is a canonical ribbon structure $\theta$
on $\B_{int}$ with respect to which categorical dimensions coincide with the Frobenius-Perron dimensions (this is true regardless of the existence of a ribbon structure on $\B$). 

\begin{lemma}
\label{thetaX=1}
Let $X$ be an object of $\B$.
\begin{enumerate}
    \item[(i)]  A symmetric object $X\in \B$ is isotropic if and only if $\theta_X=\id_X$.   
    \item[(ii)] If $X$ is a projectively isotropic simple object then $c_{X,X}^2 =\theta_X^2 \,\id_{X\ot X}$.
\end{enumerate}
\end{lemma}
\begin{proof}
(i) This is a well-known fact of symmetric fusion categories.

(ii)  We have $\theta_{X \ot X^*} = (\theta_X \ot \theta_{X^*}) c_{X^*,X}c_{X,X^*}$.  
But $\theta_{X \ot X^*} = 1$ by (i), so the statement follows by \cite[Lemma 3.15(ii)]{DGNO2}.
\end{proof}

\begin{lemma}
\label{lem complementary condition}
Let $X,\, Y$ be  simple objects of $\B$ centralizing each other.
Then $X\ot Y$ is an isotropic object of $\B$ if and only if $X,\,Y$ are projectively isotropic
and $\theta_X \theta_Y=1$.
\end{lemma}
\begin{proof}
The hexagon axioms of braiding imply that $c_{X\ot Y, X\ot Y}^2$ is  a conjugate of $c_{X,X}^2\ot c_{Y,Y}^2$.
So $X\ot Y$ is symmetric if and only if $X,Y$ are projectively isotropic and the scalars $c_{X,X}^2$ and  $c_{Y,Y}^2$
are reciprocals of each other.  But these scalars are $\theta_X^2$ and $\theta_Y^2$, respectively, so
the result follows from Lemma~\ref{thetaX=1}.
\end{proof}

\begin{definition}
\label{def complementary pair}
Let $\B_1,\B_2$ be braided fusion categories. We will say that two objects  $X_1\in \B_1,\,  X_2\in \B_2$ 
are {\em complementary} to each other if they are projectively isotropic and satisfy $\theta_{X_1}\theta_{X_2}=1$.
We will say that $X_2$ is {\em complementary} to $\B_1$ if it is complementary to some $X_1\in \B_1$ 
and {\em non-complementary} to $\B_1$ otherwise.  We will say that $\B_2$ is {\em non-complementary} to $\B_1$
if every simple object of $\B_2$ other than $\be$ is non-complementary to $\B_1$. If $\B_2$ has a specified
fermion object $\delta$ (e.g., if $\B_2$ is slightly degenerate)  we will say that $\B_2$ is {\em super non-complementary} to $\B_1$ if every simple object of $\B_2$ other than $\be$ and $\delta$ is non-complementary to $\B_1$.
\end{definition}

\begin{corollary}
\label{corad of TG times any}
Let $\B$ be a TG-category and $\C$ be a braided fusion category. Then
\begin{equation}
    \Corad(\B \bt \C) = \B \bt \tilde{\C},
\end{equation}
where $\tilde{\C}$ is the fusion subcategory of $\C$ generated by 
simple objects $Y$ complementary to $\B$.
\end{corollary}

\subsection{A canonical decomposition of a reductive category}
\label{subsect :can decomp}

Recall that an  {\em Ising} braided fusion category  $\mathcal{I}$ is a non-pointed braided fusion category of dimension $4$.  

\begin{definition}
\label{types of decomposition}
Let $\B$ be a reductive braided fusion category.  We say that $\B$ is 
\begin{enumerate}
    \item[] of type I if  $\Corad(\B)$ is non-degenerate, 
    \item[] of type II  if  both $\B$ and $\Corad(\B)$  are slightly degenerate, 
    \item[] of type III if $\B$ is non-degenerate and $\Corad(\B)$ is slightly degenerate. 
\end{enumerate}
By Remarks~\ref{reductive imply almost degenerate} and \ref{basic corad},  any reductive braided fusion category belongs to one of the above types.
\end{definition}

\begin{remark}
A reductive braided fusion category $\B$ of type I is a tensor product of a TG-category $\Corad(\B)$ and 
an anisotropic category $\Corad(\B)'$, see Corollary~\ref{saturated part of primitive}.  On the other hand,
let $\B_1$ be a reductive TG-category, $\B_2$ be an anisotropic category, and consider $\B =\B_1\bt \B_2$.
Then $\B_1 \subseteq \Corad(\B)$, but, in general, $\B_1 \neq \Corad(\B)$.
For example, take $\B_1=\Z(\Rep(\mathbb{Z}/2\mathbb{Z}))$ and $\B_2$ an Ising category. 
Then $\Corad(\B_1\bt \B_2) = \Z(\Rep(\mathbb{Z}/2\mathbb{Z})) \bt \sVect$, cf.\ Corollary~\ref{corad of TG times any}.
In Theorem~\ref{canonical decomposition} we will provide a decomposition of $\B$ in terms of its coradical and anisotropic parts.
\end{remark}

Recall that a slightly degenerate braided fusion category $\B$ admits a minimal non-degenerate extension  \cite{JFR}.
Such an extension is faithfully  $\mathbb{Z}/2\mathbb{Z}$-graded with $\B$ being the trivial component.  There are $16$ 
braided equivalence classes of such extensions, see,  e.g., \cite[Section 5.3]{DNO}

A slightly degenerate braided fusion category $\B$ is {\em split} if  $\B\cong \tilde{\B} \bt \sVect$,
where $\tilde{\B}$ is non-degenerate (the braided equivalence class of $\tilde{\B}$ is an invariant of $\B$).
Otherwise, we say that $\B$ is  {\em non-split}.

\begin{theorem}
\label{canonical decomposition}  
Let $\B$ be a reductive braided fusion category.
\begin{enumerate}
\item[(1)] $\B$ is of type I if and only if $\B\cong \B_1 \bt \B_2$, where $\B_1$ is a non-degenerate reductive TG-category and $\B_2$ is an anisotropic category non-complementary to $\B_1$.
In this case, $\B_1\cong \Corad(\B)$ and $\B_2\cong \Corad(\B)'$.
\item[(2)] $\B$ is of type II if and only if $\B\cong \B_1 \bt_{\sVect} \B_2$, where $\B_1$ is a reductive slightly degenerate TG-category and $\B_2$ is a slightly degenerate anisotropic category  super non-complementary to $\B_1$.
In this case, $\B_1\cong \Corad(\B)$ and $\B_2\cong \Corad(\B)'$.
\item[(3)]  $\B$ is of type III if and only if it is  a minimal non-degenerate extension
of a type II reductive category $\B_0$ such that  either $\B_0$ is non-split, or $\B_0\cong \tilde{\B}_0 \bt \sVect$ 
is split  and $\B \cong \tilde{\B}_0 \bt \C$, where $\C$ is a  minimal non-degenerate extension of $\sVect$ super 
non-complementary to $\B_0$. In this case, $\B_0 =\Corad(\B)\vee \Corad(\B)'$.
\end{enumerate}
\end{theorem}
\begin{proof}
(1) Since $\Corad(\B)$ is non-degenerate, we have $\B\cong \Corad(\B) \bt \Corad(\B)'$.  
It follows from  Lemma~\ref{lem complementary condition} that  $\Corad(\B)'$ must be non-complementary to 
$\Corad(\B)$. Conversely, given categories $\B_1$ and $\B_2$
satisfying the hypothesis, $\Corad(\B_1\bt \B_2)=\B_1$ by 
Corollary~\ref{corad of TG times any} and $\Rad(\B_1\bt \B_2) = \Rad(\B_1)=\Vect$.

(2) We have $\B\cong \Corad(\B) \bt_{\sVect} \Corad(\B)'$ by \cite[Proposition 4.3]{DNO}.
Hence, every simple object of $\B$ is isomorphic to $X\bt Y$, where $X\in \Corad(\B)$
and $Y\in \Corad(\B)'$, so the argument from part (1) applies  here as well.

(3) When $\B$ is of type III,  $\B_0=\Corad(\B) \vee  \Corad(\B)'$ is a type II reductive 
slightly degenerate  subcategory of $\B$ of index $2$,  so $\B$ is a minimal extension of $\B_0$.  
In the opposite direction, let us determine when a minimal extension $\B$ of a given 
type II reductive category $\B_0$ satisfies $\Corad(\B)=\Corad(\B_0)$ (in which case
$\B$ is automatically reductive).  Note that $\Rad(\B)$ is a (possibly
trivial) $\mathbb{Z}/2\mathbb{Z}$-graded extension of $\Rad(\B_0)=\Vect$, so either
$\Rad(\B) = \Vect$ (and $\B$ is reductive) or $\Rad(\B) =\Rep(\mathbb{Z}/2\mathbb{Z})$. 
The latter situation is only possible when the non-trivial homogeneous component of $\B$ contains 
an invertible object, i.e., when $\B_0\cong \tilde{\B}_0 \bt \sVect$ \cite{JFR}. 
In this case $\B= \tilde{\B}_0 \bt \C$, where $\C$ is a minimal extension of $\sVect$. 
By Corollary~\ref{corad of TG times any}, $\Corad(\B)=\Corad(\B_0)$ if and only if
$\C$ is super non-complementary to $\B$.
\end{proof}

\begin{remark}
\label{regarding factor C}
The direct factor $\C$ that appears in the split case  of Theorem~\ref{canonical decomposition}(3)
must be a non-degenerate anisotropic braided fusion category of the Frobenius-Perron dimension $4$ with exactly one fermion. There are $14$ such categories: $8$ Ising categories and $6$ pointed categories corresponding to metric groups
in lines 3 and 4 of Table~\ref{table-2}. Note that Ising categories are super non-complementary to any $\B_0$,
while for the pointed ones this property depends on the choice of $\B_0$.
\end{remark}


\section{Reductive pointed braided fusion categories}
\label{Sect reductive pointed}

Recall that pointed braided fusion categories correspond to pre-metric groups, i.e., pairs
$(A,\, q)$, where $q:A\to k^\times$ is a quadratic form.  These categories are denoted $\C(A,\, q)$.
The goal of this Section is to classify pre-metric groups $(A,\, q)$ such that $\C(A,\, q)$ is reductive. 

\subsection{Isotropically generated metric groups}

We will say that a pre-metric group $(A,\, q)$ is  {\em isotropically generated} if 
$A$ is  generated by its {\em isotropic elements}, i.e., by those $x\in A$  for which $q(x)=1$.
The set of such elements is called the {\em light cone} in the literature.

Clearly, $(A,\,q)$ is isotropically generated if and only if $\C(A,\, q)$ is  a TG-category.

\begin{example}
\label{ex hyperbolic metric group}
For any finite Abelian group $A$ let 
\begin{equation}
\label{hyperbolic form h}    
h: A \oplus \widehat{A} \to k^\times, \qquad h(a,\, \phi) =\langle \phi,\, a \rangle,\quad a\in A,\, \phi \in \widehat{A},
\end{equation}
denote the canonical hyperbolic quadratic form on $A$. The {\em hyperbolic} metric group 
$(A \oplus \widehat{A},\, h)$ is isotropically generated since $A$ and $\widehat{A}$ 
are isotropic subgroups of $A \oplus \widehat{A}$.  
\end{example}

\begin{theorem}
\label{pointed saturated}
Let $(A,\,q)$ be a metric group and let $n$ be the exponent of $A$. 
Then $(A,\,q)$ is  isotropically generated if and only if 
the following two conditions are satisfied:
\begin{enumerate}
\item[(i)]  $q^n=1$, and 
\item[(ii)] $(A,\,q)$ contains a hyperbolic metric subgroup $(\mathbb{Z}/n\mathbb{Z}\times \mathbb{Z}/n\mathbb{Z},\, h)$.
\end{enumerate}
\end{theorem}
\begin{proof}
We may assume that $A$ is a $p$-group. Let
\[
B(x,\,y) =\frac{q(x+y)}{q(x)q(y)},\qquad x,y\in A,
\]
be the non-degenerate symmetric bilinear form on $A$ associated to $q$.

Suppose that $(A,\,q)$ is isotropically generated. 
Every $x\in A$ can be written as  $x=x_1+\cdots +x_k$
with $q(x_i) =1,\, i=1,\dots k$. Therefore, $q(x) = \prod_{1 \leq i<j\leq k}\, B(x_i,\, x_j)$.  But each factor $B(x_i,\, x_j)$ is an $n$th root of $1$ and so $q(x)^n=1$,
proving (i).

Note that $A$ must contain an isotropic element $a$ of order $n$. Indeed, isotropic elements of $A$ of order  less than $n$ generate
a subgroup of exponent less than $n$, i.e., a proper subgroup of $A$. We claim that there is $c\in A$  such that 
\begin{equation}
\label{q=1 and Bac primitive}
\text{$q(c)=1$ and $B(a,\,c)$ is a primitive $n$th root of $1$.}
\end{equation}
To see this, let $n_1=n/p$ and suppose that  $B(a,x)^{n_1}=1$ for all $x\in A$. We have
$B(n_{1}a,\,x)  =  B(a,\,x)^{n_1}   =1$, which contradicts the non-degeneracy of $B$. 
So there is $c'\in A$ such that $B(a,\,c')$  is a primitive $n$th root of $1$. Hence,  the order of $c'$ in $A$ 
is also $n$.
Let us choose $k\in \mathbb{Z}$ such that $B(a, c')^k \, q(c')=1$ and set $c=ka+c'$. Then $c$
satisfies \eqref{q=1 and Bac primitive} and 
the metric group $(\langle a,\, c \rangle, \, q|_{\langle a,\, c \rangle})$ is isomorphic to
$(\mathbb{Z}_n\times \mathbb{Z}_n,\, h)$.  This proves (ii).

Now suppose that $(A,\,q)$ satisfies conditions (i) and (ii). Then 
\[
(A,\,q) \cong (H,\, h) \oplus (A_1,\, q_1),
\]
where  $(H,\, h) := (\mathbb{Z}_n\times \mathbb{Z}_n,\, h)$ and $(A_1,\, q_1)$ is  metric group such that $q_1^n=1$.
To prove that $(A,q)$ is isotropically generated, note that
the image of $h$ consists of all $n$th roots of unity. Therefore, $(A_1,\, q_1)$ is complementary to
$(H,\, h)$ in the sense of Definition~\ref{def complementary pair} (here we use the same terminology for metric groups and braided fusion categories).  Hence, $(A,\,q)$ is isotropically generated by Corollary~\ref{corad of TG times any}.
\end{proof}

\begin{corollary}
\label{simple form of isogen}
Let $p$ be an odd prime. Any isotropically generated metric $p$-group $(A,\,q)$ of exponent $p^e$ can be written as
\[
(A,\,q) = (\mathbb{Z}/p^e\mathbb{Z}\times \mathbb{Z}/p^e\mathbb{Z},\, h) \,\bigoplus\,
\left\{ \text{sum of metric $p$-groups of exponent $\leq p^e$}  \right\}.
\]
Any isotropically generated metric $2$-group $(A,\,q)$ of exponent $2^e$ can be written as
\begin{eqnarray*}
(A,\,q) &=& \lefteqn{ (\mathbb{Z}/2^e\mathbb{Z}\times \mathbb{Z}/2^e\mathbb{Z},\, h)^{\oplus n}  
\,\bigoplus\, (\mathbb{Z}/2^e\mathbb{Z}\times \mathbb{Z}/2^e\mathbb{Z},\, f)^{\oplus \epsilon}   }\\
& & \,\bigoplus\, \left\{ \text{sum of metric $2$-groups of exponent $< 2^e$}  \right\}, 
\qquad n \geq 1,\,\epsilon\in \{0,\,1\},
\end{eqnarray*}
where the forms $h$ and $f$ were defined in \eqref{hzeta} and \eqref{fzeta}.
\end{corollary}
\begin{proof}
Condition (i) of Theorem~\ref{pointed saturated}   is automatically satisfied for any odd $p$,
which implies the first decomposition. The second one follows by Lemma~\ref{qn=1 homogeneous}.
\end{proof}

\begin{corollary}
\label{isogen solutions qx=1}
Let $(A,\,q)$ be an isotropically generated metric group and let $n$ be the exponent of $A$.
For any $n$th root of unity $\zeta\in k^\times$ there is $x\in A$ such that $q(x)=\zeta$.
\end{corollary}
\begin{proof}
This is easily seen to be true for the hyperbolic group    
$(\mathbb{Z}/p^n\mathbb{Z}\times \mathbb{Z}/p^n\mathbb{Z},\, h)$, so the result follows
from Theorem~\ref{pointed saturated}.
\end{proof}

\begin{lemma}
\label{coradical absorbtion}
Let $(A,\,q)$ be an isotropically generated  metric group and  let $n$ be the exponent of $A$.
Let $(B,\,r)$ be a  pre-metric group. Define $\tilde{B}=\langle y\in B \mid  r(y)^n =1 \rangle$.
Then $A\oplus \tilde{B}$ is an isotropically generated subgroup of $(A \oplus B,\, q\oplus r)$. 
\end{lemma}
\begin{proof}
It follows from Corollary~\ref{isogen solutions qx=1} that for any $y\in \tilde{B}$ 
there is $x\in A$ such that $q(x)r(y)=1$ and so $(x,y)$ is an isotropic element of $A\oplus B$.
Therefore, $A\oplus \tilde{B}$ is isotropic.
\end{proof}



\subsection{Reductive metric $p$-groups.}
\label{subsect: Reductive metric $p$-groups}

\begin{definition}
A pre-metric group is {\em reductive} if the intersection of its maximal isotropic subgroups
is trivial.
\end{definition}

Equivalently, a pre-metric group $(A,\,q)$ is reductive if and only if the braided fusion category $\C(A,\,q)$ is reductive.

\begin{proposition}
\label{odd dichotomy}
Let $p$ be an odd prime. A metric $p$-group $(A,\,q)$ is reductive if and only it is either isotropically generated or
anisotropic.
\end{proposition}
\begin{proof}
Since $|A|$ is odd, the reductive pointed fusion category $\C(A,\,q)$ is of type I, see Definition~\ref{types of decomposition}. By Theorem~\ref{canonical decomposition}(1), 
 $(A,\,q) = (A_1,\, r) \bigoplus (A_2,\, s)$,
where $(A_1,\, r)$ is an isotropically generated metric $p$-group and $(A_2,\,s)$ is an anisotropic metric $p$-group non-complementary to $(A_1,\, r)$. If  $(A_1,\, r)$ is non-trivial it contains a hyperbolic summand by Theorem~\ref{pointed saturated}. So the values of $r$ include all $p$th roots of $1$. But the values of $s$
are also $p$th roots of $1$ and so if $(A_2,\,s)$ is non-trivial, it must be complementary to $(A_1,\, r)$, a contradiction.
\end{proof}

\begin{lemma}
\label{split summand}
Let $(A,\,q)$ be a slightly degenerate isotropically generated pre-metric $2$-group.
Then there is an isotropically generated metric $2$-group $(\tilde{A},\,\tilde{q})$ and
an orthogonal isomorphism 
\begin{equation}
\label{split iso}
(A,\,q) \cong (\tilde{A},\,\tilde{q}) \oplus (\mathbb{Z}/2\mathbb{Z},\, q_{-1}),
\end{equation}
where $(\mathbb{Z}/2\mathbb{Z},\, q_{-1})$ is the unique non-isotropic degenerate  metric group of order $2$, cf.\ \eqref{qzeta}.
\end{lemma}
\begin{proof}
By \cite[Corollary A.19]{DGNO1}, $(A,\,q)$ splits into an orthogonal direct sum as in \eqref{split iso}, so we only need to show that $(\tilde{A},\,\tilde{q})$ can be chosen to be isotropically generated. 
If the exponent of $A$ is $\geq 4$
then $(A,\,q)$ (and, hence, $(\tilde{A},\,\tilde{q})$) contains a hyperbolic summand of the same exponent.

If the exponent of $A$ is $2$ then $(\tilde{A},\,\tilde{q})$ is an orthogonal sum of metric groups
\eqref{hzeta}, \eqref{fzeta} by Lemma~\ref{qn=1 homogeneous}. Since
\begin{equation}
\label{E+sVect = F+sVect}
(\mathbb{Z}/2\mathbb{Z}\times \mathbb{Z}/2\mathbb{Z},\, h) \oplus (\mathbb{Z}/2\mathbb{Z},\, q_{-1})
\cong (\mathbb{Z}/2\mathbb{Z}\times \mathbb{Z}/2\mathbb{Z},\, f) \oplus (\mathbb{Z}/2\mathbb{Z},\, q_{-1})
\end{equation}
we can assume that $(\tilde{A},\,\tilde{q})$ has a hyperbolic summand.

By Theorem~\ref{pointed saturated}, in both cases the metric group $(\tilde{A},\,\tilde{q})$ 
is isotropically generated.
\end{proof}

\begin{theorem}
\label{reductive metric 2-groups classification}
For a reductive pre-metric $2$-group exactly one of the following is true:
\begin{enumerate}
    \item[(1)] it is isotropically generated,
    \item[(2)] it is non-trivial anisotropic (see Table~\ref{table-2}),
    \item[(3)] it is an orthogonal  sum of an isotropically generated metric $2$-group of exponent $2$ 
    and a non-trivial anisotropic pre-metric $2$-group with at most one fermion (such groups are labelled by ${}^*$
    in Table~\ref{table-2}).    
\end{enumerate}
\end{theorem}
\begin{proof}
All we need to do is to classify reductive pre-metric $2$-groups  $(B,\,r)$
that are neither isotropically generated nor anisotropic. Let $(A,q)$ be the maximal
isotropically generated subgroup (i.e., the coradical) of $(B,\,r)$.

If $(A,q)$ is non-degenerate then 
\[
(B,\,r) = (A,\,q) \oplus (A',\,q'),
\]
where  $(A',\,q')$ is anisotropic and non-complementary to $(A,\,q)$, see Corollary~\ref{corad of TG times any}.
By Lemma~\ref{coradical absorbtion}, $q'(x)^n \neq 1$ for any non-zero $x\in A$, 
where $n$ is the exponent of $A$.  From Table~\ref{table-2}
we see that this is only possible when $n=2$ and $(A',\,q') =(\mathbb{Z}/2\mathbb{Z},\, q_i),\, i^2=-1$. 

If $(A,q)$ is slightly degenerate then using Lemma~\ref{split summand} we have
\[
(B,\,r) = (\tilde{A},\,\tilde{q}) \oplus (A',\,q'),
\]
where $(\tilde{A},\,\tilde{q})$ is an isotropically generated 
metric $2$-group such that \eqref{split iso} holds and  
$(A',\,q')$ is an anisotropic pre-metric group containing a (necessarily unique)
fermion. The same argument as above shows that the exponent of $A$ is $2$, so the statement follows.
\end{proof}

\begin{remark}
Equation~\eqref{E+sVect = F+sVect} implies that the choice of an isotropically generated summand in
Theorem~\ref{reductive metric 2-groups classification}(3) is non-unique. But the anisotropic summand is unique as it is the orthogonal complement of the coradical.
\end{remark}


\section{Classification results}
\label{sect classification results}

\subsection{Reductive $p$-categories of small dimension}
\label{Sect reductive small}

Let $p$ be a prime. We say that a braided fusion  category $\B$ is a $p$-category if $\FPdim(\B) =p^n$ for some integer $n\geq 0$. Such categories are weakly group-theoretical by \cite{ENO3}.

\begin{lemma}
\label{Lemma TGof dim<p5}
Let $\B$ be a reductive category with $\FPdim(\B)=p^{n},\,n\leq 5$. Then 
$\Corad(\B)$ is pointed.    
\end{lemma}
\begin{proof}
Let $\E \subset \Corad(\B)$ be a maximal Tannakian subcategory. 

If $\Corad(\B)$
is non-degenerate, then $\FPdim(\E)^2 \leq \FPdim(\E)\FPdim(\E') = \FPdim(\B)\leq p^5$ and so $\FPdim(\E)\leq p^2$. This implies that $\E$ is pointed. Since $\Corad(\B)$
is generated by Tannakian subcategories of $\B$, it is pointed as well. 

If $p=2$ and $\Corad(\B)$ is slightly degenerate, then any maximal {\em symmetric}
subcategory $\E \subset \B$ satisfies $\FPdim(\E)^2 \leq \FPdim(\B) \FPdim(\Z_{sym}(\B)) = 64$ and so $\FPdim(\E)\leq 8$. But $\E$ must contain $\Z_{sym}(\B)=\sVect$,
and so a maximal {\em Tannakian} subcategory of $\B$ has FP dimension at most $4$
and, hence, is pointed. This implies that $\Corad(\B)$ is pointed. 
\end{proof}

\begin{proposition}
\label{Prop: pointed x ising}
Let $\B$ be a reductive category with $\FPdim(\B)=p^{n},\,n\leq 5$. Then either
$\B$ is pointed or $p=2$ and $\B$ is  a product of a pointed category and an Ising category. 
\end{proposition}
\begin{proof}
Recall three types of reductive categories introduced in Definition~\ref{types of decomposition}.

If $\B$ is of type I or II 
then by Theorem~\ref{canonical decomposition}(1,2), $\B$ is generated by  
$\Corad(\B)$ and $\Corad(\B)'$.
Note that $\Corad(\B)$ is pointed by Lemma~\ref{Lemma TGof dim<p5} and $\Corad(\B)'$
is non-degenerate anisotropic. By \cite{Na}, $\Corad(\B)'$ is either pointed
or is a  
product of a pointed braided category and an Ising category (for $p=2$), 
which implies the statement in this case. 

If $\B$ is of type III (so $p=2$) then $\Corad(\B)'$ must be pointed. Indeed,
$\Corad(\B)'$ is slightly degenerate anisotropic in this case. If it is not 
pointed, then it contains an Ising subcategory  and a fermion centralizing it,
which contradicts its anisotropy. Thus, the subcategory $\B_0 =\Corad(\B)\vee \Corad(\B)' \subset \B$ is pointed. Therefore, $\B_0$ is split, i.e., $\B_0 =\tilde{\B}_0 \bt \sVect$ and by Theorem~\ref{canonical decomposition}(3),
$\B= \tilde{\B}_0 \bt \C$, where $\tilde{\B}_0$ is pointed and $\C$ is a minimal extension of $\sVect$ (and so $\C$ is either pointed or Ising).
\end{proof}

\begin{corollary}
\label{up to p5 pointed}
An integral reductive category of dimension $p^n,\, n\leq 5,$  is pointed.  
\end{corollary}

It is straightforward to list equivalence classes of pointed reductive categories
using results of Section~\ref{subsect: Reductive metric $p$-groups}. It is similarly not hard to list reductive categories that are products of Ising categories and pointed ones. Below we give 
a classification of {\em integral non-pointed} reductive categories of the Frobenius-Perron dimension $p^6$. In view of Corollary~\ref{up to p5 pointed} this  is the smallest dimension in which such categories exist.

For any prime $p$ and $n\geq 1$ let $E_{p^n}$ denote the elementary Abelian $p$-group of order $p^n$.

\begin{lemma}
\label{lem B=ZEp3w}
Let $\B$ be an integral non-pointed reductive category with $\FPdim(\B)=p^6$.
Then $\B\cong \Z(\Vect_{E_{p^3}}^\omega)$ for some $\omega \in H^3(E_{p^3},\,k^\times)$.
\end{lemma}
\begin{proof}
This statement is equivalent to $\B$ being non-degenerate and containing a Lagrangian subcategory equivalent to $\Rep(E_{p^3})$. Note that $\Corad(\B)'$ is anisotropic integral and, hence, pointed. Let $\E\in \MT_{max}(\B)$. We claim that $\FPdim(\E)=p^3$. Indeed, if $\FPdim(\E)\leq p^2$ then $\E$ is pointed and so both $\Corad(\B)$ and $\B_0=\Corad(\B) \vee \Corad(\B)'$ are pointed. We have  either $\B= \B_0$ or $\B_0$ is slightly degenerate, and $\B$ is its minimal extension. In the latter case $\B$ must contain an Ising subcategory, contradicting its integrality. On the other hand, there can be no symmetric subcategories $\E \subset \B$ with $\FPdim(\E)\geq p^4$, since this would contradict $\B$ being almost non-degenerate.
The last observation also implies that  $\B$ is non-degenerate, since otherwise
$\E \vee \Z_{sym}(\B)$ is symmetric of dimension $\geq p^4$.

It follows from \cite{DGNO2}, that $\B\cong \Z(\Vect_G^\omega)$ for some group $G$
of order $p^3$. Since $\B$ is reductive metabolic, it is a TG-category by Proposition~\ref{TG properties}(ii). Since $\B$ is non-pointed, it must contain non-pointed Lagrangian subcategories, so $G$ can be chosen non-Abelian. For any such $G$  its center is $Z(G)\cong E_p$ and 
the group of invertible objects in $\Rep(G)$ is $\widehat{G}\cong E_{p^2}$.  Furthermore, the group $\Inv(\B)$ of invertible objects of $\B$ fits into the exact sequence
\begin{equation}
1\to \widehat{G} \to \Inv(\B)  \to Z(G),
\end{equation}
see, e.g., \cite{GP, MN}. In particular, $|\Inv(\B)|\leq p^3$.  In fact,
$|\Inv(\B)|= p^3$, since otherwise $\B_{pt}$ is contained in every Lagrangian subcategory of $\B$, which is impossible since $\B$ is reductive.  For the same reason $\B_{pt}$
must contain more than one Tannakian subcategory equivalent to $\Rep(E_{p^2})$. Therefore, 
$\B_{pt}= \C(E_{p^3},\, q)$, where the pre-metric group   $(E_{p^3},\, q)$ is isotropically generated and the set of its isotropic subgroups of order $p^{2}$ 
has a trivial intersection. This is only possible when $q=1$, i.e., $\B_{pt}\cong 
\Rep(E_{p^3})$.
\end{proof}

\begin{lemma}
\label{lemma classes vs orbits}
Non-pointed braided fusion categories $\Z(\Vect_{E_{p^3}}^{\omega_1})$
and $\Z(\Vect_{E_{p^3}}^{\omega_2})$ are equivalent if and only if 
the cohomology classes of $\omega_1,\omega_2$ lie in the same 
orbit of $GL_3(\mathbb{F}_p)$.
\end{lemma}
\begin{proof}
Each of these categories has a unique pointed Lagrangian subcategory 
equivalent to $\Rep(E_{p^3})$ \cite{MN} and so a braided equivalence $\alpha$
between them restricts to a braided equivalence between these subcategories.
By \cite[Theorem 3.3]{MN}, this is equivalent to $\alpha$ being induced from
a tensor equivalence between $\Vect_{E_{p^3}}^{\omega_1}$ and 
$\Vect_{E_{p^3}}^{\omega_2}$. Such an equivalence exists if an only if
$\omega_2$ is cohomologous to $f^*\omega_1$ for some automorphism $f$ of $E_{p^3}$.
\end{proof}

\begin{proposition}
\label{Prop p+1 classes}
For any prime $p$
there are precisely $p+1$ braided equivalence classes of integral non-pointed reductive categories of the Frobenius-Perron dimension $p^6$.
\end{proposition}
\begin{proof}
In view of Lemmas~\ref{lem B=ZEp3w} and \ref{lemma classes vs orbits}, 
the number in question is equal to
the number of $GL_3(\mathbb{F}_p)$-orbits of $\omega\in H^3(E_{p^3},\,k^\times)$
such that $\Z(\Vect_{E_{p^3}}^\omega)$ is reductive.

Suppose that $p$ is odd. 
There is a $GL_3(\mathbb{F}_p)$-equivariant isomorphism 
\[
H^3(E_{p^3},\,k^\times) \xrightarrow{\sim} \wedge^3 (\mathbb{F}_p^3)\, \bigoplus \,
\mathsf{S}^2 (\mathbb{F}_p^3): \omega \mapsto (\omega_{alt},\, \omega_{sym}).
\]
The action of $A\in GL_3(\mathbb{F}_p)$ is given by
\[
A (\omega_{alt},\, \omega_{sym}) = (\det(A) \omega_{alt},\, A^t \omega_{sym} A),
\]
where we identify $\omega_{sym}$ with a symmetric matrix and $A^t$ denotes the transpose of $A$. The category $\Z(\Vect_{E_{p^3}}^\omega)$ is non-pointed if and only if $\omega_{alt}\neq 0$ (i.e., $\omega_{alt}=\lambda \det, \, \lambda\in 
\mathbb{F}_p^\times$). In this case $\Z(\Vect_{E_{p^3}}^\omega)_{pt}\cong \Rep(E_{p^3})$,
see \cite[Section 5.1]{MN}. Such a category $\Z(\Vect_{E_{p^3}}^\omega)$ is reductive
if and only if  $E_{p^3}$ is generated by its subgroups $N$ such that $\omega \in \Omega(E_{p^3};N)$ \eqref{OmegaGN}. This is equivalent to $E_{p^3}$ being generated by its cyclic subgroups $C$ such that $\omega_{sym}|_C =0$. In other words,  $E_{p^3}$  is an isotropically generated metric group. There are precisely $4$ congruence classes of symmetric $3\times 3$ matrices $\omega_{sym}$ with this property
\begin{equation}
\label{eqn list of 4 matrices}
A_0= \begin{pmatrix}
0 & 0 & 0 \\ 0 & 0 & 0 \\ 0 & 0 & 0
\end{pmatrix},\,
A_1= \begin{pmatrix}
0 & 0 & 0 \\ 0 & 0 & 1 \\ 0 & 1 & 0
\end{pmatrix},\,
A_2=\begin{pmatrix}
1 & 0 & 0 \\ 0 & 0 & 1 \\ 0 & 1 & 0
\end{pmatrix},\,\, \text{and}\,
A_3=\begin{pmatrix}
\zeta & 0 & 0 \\ 0 & 0 & 1 \\ 0 & 1 & 0
\end{pmatrix},
\end{equation}
where $\zeta$ is any quadratic non-residue modulo $p$.  A description of $GL_3(\mathbb{F}_p)$-orbits of the cohomology classes  $(\lambda\det,\, \omega_{sym}),\, \lambda \neq 0$, is contained in \cite{MCU}. A complete set of representatives of these orbits is
\begin{equation}
(\det,\, A_0),\,  (\det,\, A_1),\,  (\pm \mu \det,\, A_2),\ \text{and}\ (\pm \mu \det,\, A_3),\, \mu \in \mathbb{F}_p^\times.
\end{equation}
Thus, altogether there are  $1+1+ \frac{p-1}{2} + \frac{p-1}{2} =p+1$ orbits.

Now let $p=2$. We can use the following description of the 
$GL_3(\mathbb{F}_2)$-orbits on
$H^3(E_8,\, k^\times)$ from \cite{GMN}. Namely, 
there is a bijection between $H^3(E_8,\, k^\times)$
and subsets of the set $S$ of order $2$ subgroups of $E_8$ via
\[
\omega \mapsto \{ C \in S \mid \text{ the class $\omega|_{C} \in H^3(C,k^\times)$ is non-trivial} \}.
\]
The corresponding subset of $S$ is called the support of $\omega$ and denoted
$\text{supp}(\omega)$.  The {\em weight} of $\omega$ is the cardinality $|\text{supp}(\omega)|$ of its support. The category $\Z(\Vect_{E_8}^\omega)$ is non-pointed
if and only if the  weight of $\omega$ is $1,3,5$, or $7$.
The cohomology classes with the weight $1,5$, and $7$  form three orbits, 
one for each value of the weight, while those of weight $3$ form two orbits,
see \cite{GMN}.  By Remark~\ref{rem nilpotency}, the category $\Z(\Vect_{E_8}^\omega)$  is reductive
if and only if subgroups in $S\backslash \text{supp}(\omega)$ generate $E_8$, 
i.e., when the weight of $\omega$ is $1$ or $3$. Thus, there are $3$ orbits  in this case.
\end{proof}

\begin{remark}
Braided fusion categories from Proposition~\ref{Prop p+1 classes} are metabolic in the sense of Definition~\ref{def metabolic}.    
\end{remark}

\subsection{Braided fusion categories of small dimension}
\label{sect braided  small}

Here we apply the results obtained in the previous sections to the classification of non-degenerate braided fusion categories whose Frobenius-Perron 
dimension is a product of at most four prime numbers. There is a substantial literature on this subject (see, e.g., \cite{ACRW, CP, BCHKNNPR, DoN, EGO}). A typical approach is to show that all categories in a given class (for example, those with a fixed number of simple objects or a prescribed Frobenius-Perron dimension) are (weakly) group-theoretical. In such cases, their construction reduces to questions about finite groups and their cohomology. On the other hand, results that classify these categories up to braided equivalence are rather scarce. For example, as we pointed out in Remark~\ref{not an invariant}, describing them in terms of twisted Drinfeld doubles does not lead to a satisfactory classification.

The following algorithm provides a method for enumerating equivalence classes of non-degenerate braided fusion categories of a given Frobenius-Perron dimension $d$.
\begin{enumerate}
\item[(1)] Find all pairs $(\B,\, G)$, where $\B$ is a non-degenerate reductive braided fusion category and $G$ is a finite group such that 
$d = \FPdim(\B)\,|G|^2$.
\item[(2)]  Classify all actions of $G$ on $\B$, i.e., monoidal functors $G \to \Aut^{br}(\B)$, up to automorphism of $G$ and 
conjugation by a braided autoequivalence of $\B$.
\item[(3)]  For each such action, determine its lifts to Tannakian primitive monoidal $2$-functors $F: G \to \uuPic(\B)$ using the obstruction theory of \cite{DN2, ENO2}.
By Theorem~\ref{CGT is a complete invariant}, equivalence classes of such lifts are in bijection with equivalence classes of non-degenerate braided fusion categories  
$\tilde{\B}$ satisfying $G_{\tilde{\B}} = G$ and $\Mantle(\tilde{\B}) = \B$. It follows from \cite[Corollary 2.27]{DN2} that these equivalence classes are parameterized by 
the orbits of the group $\Aut(G)$ acting on $\text{Coker}\!\left(H^1(G,\, \Inv(\B)) \xrightarrow{c_F} H^3(G,\, k^\times)\right)$,
where $c_F$ is a certain canonical homomorphism associated with $F$.
\end{enumerate}

\begin{remark}
\begin{enumerate}
\item[(a)] The feasibility of the above algorithm depends on our ability to classify reductive categories of a given dimension. It works best in situations where the anisotropic part of $\B$ can be controlled, in particular when $\B$ is weakly group-theoretical (and hence is nilpotent by Proposition~\ref{red WGT nil}). This property holds for all known examples of weakly integral braided fusion categories. 
In particular, it is satisfied for fusion categories of small Frobenius-Perron dimension listed in Proposition~\ref{prop: list} below.
\item[(b)]  By contrast, our theory does not shed light on anisotropic braided fusion categories. It might still be possible to
localize such categories by means of \'etale algebras, but they can no longer be reconstructed as gaugings.
\item[(c)] In practice, the Tannakian primitive property of a monoidal $2$-functor $G \to \uuPic(\B)$ may be verified by establishing that the corresponding gauging 
$\tilde{\B}$ satisfies $\Rad(\tilde{\B}) = \Rep(G)$.
\end{enumerate}
\end{remark}

Since pointed braided fusion categories correspond to metric groups, their classification is a 
straightforward but tedious bookkeeping exercise using relations found in \cite{Mir,W}. In Example~\ref{prop: list} below we focus on non-pointed 
braided fusion categories of small dimension. 

We will use the following basic braided fusion categories.

An {\em Ising} braided fusion category  $\mathcal{I}_\eta$ is a non-pointed braided fusion category of dimension $4$.  
These  are parameterized by primitive  $16$th roots of unity $\eta$  \cite[Appendix B]{DGNO2}. 

A {\em  Tambara-Yamagami} fusion category $\TY(A,\, \chi,\, \tau)$ \cite{TY}
is a $\mathbb{Z}/2\mathbb{Z}$-graded category  with a unique non-invertible simple object. They are parameterized by triples $(A,\, \chi,\, \tau)$,
where $A$ is an Abelian group, $\chi:A \times A \to k^\times$  is a symmetric non-degenerate bicharacter, and $\tau$ is an element of  a $\mathbb{Z}/2\mathbb{Z}$-torsor. 

A {\em metaplectic}  braided fusion category  \cite{ACRW, BPR} is a $\mathbb{Z}/2\mathbb{Z}$-gauging of $\C(\mathbb{Z}/n\mathbb{Z},\, q)$, 
where the generator of $\mathbb{Z}/2\mathbb{Z}$ acts by inversion. 

\begin{example}
\label{prop: list}
Let us apply the above algorithm to classify pointed non-degenerate braided fusion categories whose Frobenius-Perron dimensions are products of at most four prime factors. A complete list of such categories is given in Table~\ref{table-3}. Our counting arguments are greatly simplified by the fact that the relevant cohomological obstructions -- namely, lifting a homomorphism to a monoidal functor and then lifting the latter to a monoidal $2$-functor -- vanish in all cases. Consequently, the number of monoidal $2$-functors corresponding to a given action is always $|H^2(G,\, \Inv(\B))|\, |H^3(G,\, k^\times)|$ minus the number of imprimitive extensions (Definition~\ref{def primitive functor}).

\begin{table}[t]
\begin{tabular}{|p{1.1cm}|>{\raggedright\arraybackslash}p{4.1cm}|p{0.9cm}|p{3.3cm}|>{\raggedright\arraybackslash}p{4.6cm}|p{0.3cm}|}
\hline
 ${\rm FPdim}$  & $\B=$ {\rm the mantle} & $G$ &  {\rm Generator action} & {\rm Description} & $\#$ \\
\hline
\hline
$4$ &    $\mathcal{I}_\eta$   &   $1$  &   & {\rm Ising~categories}  &   $1$ \\
\hline
$8$ & $\mathcal{I_\eta} \bt \C(\mathbb{Z}/2\mathbb{Z},\,q_{\eta^4})$    & $1$  &  &  ${\rm Ising} \bt \,\C(\mathbb{Z}/2\mathbb{Z},\, q_{i})$  & $1$ \\
\hline
$4r$ &    $\C(\mathbb{Z}/r\mathbb{Z},q_\mu)$   &   $\mathbb{Z}/2\mathbb{Z}$  & $x \mapsto -x$  &  {\rm Metaplectic}   &   $2$ \\
\hline
$16$ &               $\mathcal{I}_\eta\bt \C(\mathbb{Z}/2\mathbb{Z} \times \mathbb{Z}/2\mathbb{Z} ,h)$                                              & $1$                                        
&        & 
$\mathcal{I}_\eta\bt \C(\mathbb{Z}/2\mathbb{Z} \times \mathbb{Z}/2\mathbb{Z} ,h)$    &   $1$\\
\hline
$16$ &               $\mathcal{I}_\eta$                                              & $\mathbb{Z}/2\mathbb{Z}$                                        &    {\rm trivial}     & 
{\rm  Ising $\bt\,$ pointed}    &   $3$\\
\hline
$16$ &  $\C(\mathbb{Z}/4\mathbb{Z},\,q_\xi)$ & $\mathbb{Z}/2\mathbb{Z}$  & $x \mapsto -x$  &  $ \mathcal{I}_{\eta} \bt
\mathcal{I}_{\xi\eta^{-1}}$   &  $4$ \\
\hline
$16$ &  $\C(\mathbb{Z}/2\mathbb{Z},\,q_i)^{\bt 2}$ & $\mathbb{Z}/2\mathbb{Z}$  & $(x,y) \mapsto (y,x)$  &  $ \mathcal{I}_{\eta} \bt
\mathcal{I}_{i\eta^{-1}}$   &  $4$ \\
\hline
$16$ &  $\C(\mathbb{Z}/2\mathbb{Z} \times \mathbb{Z}/2\mathbb{Z} ,e)$ & $\mathbb{Z}/2\mathbb{Z}$  & $(x,y) \mapsto (y,x)$  &  $ \mathcal{I}_{\eta} \bt
\mathcal{I}_{-\eta^{-1}}$   &  $4$ \\
\hline
$16$ &  $\C(\mathbb{Z}/2\mathbb{Z} \times \mathbb{Z}/2\mathbb{Z} ,h)$ & $\mathbb{Z}/2\mathbb{Z}$  & $(x,y) \mapsto (y,x)$  &  $ \mathcal{I}_{\eta} \bt
\mathcal{I}_{\eta^{-1}}$   &  $4$ \\
\hline
$8r$ &    $\C(\mathbb{Z}/2r\mathbb{Z},q_{i\mu})$   &   $\mathbb{Z}/2\mathbb{Z}$  & $x \mapsto -x$  
& {\rm Metaplectic}   &   $4$ \\
\hline
$4r^2$ &  $\Vect$  & $D_r$  &    &  
 $\Rep(D^\omega(D_r)),\,\omega|_{\mathbb{Z}/r\mathbb{Z}}\neq 1$   & $4$\\
\hline
$4r^2$ &  $\C(\mathbb{Z}/r\mathbb{Z} \times \mathbb{Z}/r\mathbb{Z} ,h)$  & $\mathbb{Z}/2\mathbb{Z}$  & $(x,y)\mapsto(-x,-y)$  &  
 $\Rep(D^\omega(D_r)),\,\omega|_{\mathbb{Z}/r\mathbb{Z}}=1$   & $2$\\
 \hline
 $4r^2$ &  $\C(\mathbb{Z}/r\mathbb{Z} \times \mathbb{Z}/r\mathbb{Z} ,h)$  & $\mathbb{Z}/2\mathbb{Z}$  & 
$(x,y)\mapsto  (\pm y,\pm x)$  &   
$\Z(\TY(A,\, \chi,\, \tau))$  &$4$  \\
\hline
$4r^2$ &  $\C(\mathbb{Z}/r\mathbb{Z} \times \mathbb{Z}/r\mathbb{Z} ,e)$  & $\mathbb{Z}/2\mathbb{Z}$  & $(x,y)\mapsto(-x,-y)$  &     & $2$\\
\hline
$4r^2$ &  $\C(\mathbb{Z}/r\mathbb{Z} \times \mathbb{Z}/r\mathbb{Z} ,e)$  & $\mathbb{Z}/2\mathbb{Z}$  & $(x,y)\mapsto(\pm y, \pm x)$  
&  &$4$  \\
\hline
$4pr$ &  $\C(\mathbb{Z}/pr\mathbb{Z}, q_{\zeta\mu})$  & $\mathbb{Z}/2\mathbb{Z}$  & $(x,y)\mapsto(-x,  -y)$  & {\rm Metaplectic}  &  $2$   \\
\hline
$4pr$ &  $\C(\mathbb{Z}/pr\mathbb{Z}, q_{\zeta\mu})$  & $\mathbb{Z}/2\mathbb{Z}$  & $(x,y)\mapsto(\pm x,  \mp y)$  & 
{\rm Metaplectic} $\bt$ {\rm pointed}&  $4$   \\
\hline
$p^2r^2$, $p|r-1$ &  $\C(\mathbb{Z}/r\mathbb{Z} \times \mathbb{Z}/r\mathbb{Z} ,h)$ & $\mathbb{Z}/p\mathbb{Z}$  &   {\rm rotation of order} ${p}$ &
$\Rep(D^\omega(r\mathbb{Z}\rtimes p\mathbb{Z}))$    &      $p$\\
\hline
$p^2r^2$, $p|r+1$  &  $\C(\mathbb{Z}/r\mathbb{Z} \times \mathbb{Z}/r\mathbb{Z} ,e)$ & $\mathbb{Z}/p\mathbb{Z}$  &  {\rm rotation of order} ${p}$  &   
& $p$\\
\hline
$36$   &  $\C(\mathbb{Z}/2\mathbb{Z} \times \mathbb{Z}/2\mathbb{Z} ,e)$ & $\mathbb{Z}/3\mathbb{Z}$  &  {\rm rotation of order} ${3}$  &   
& $3$\\
\hline
\end{tabular}
\medskip
\caption{\label{table-3} Non-pointed categories of small FP dimension with the radical $\Rep(G)$ and mantle $\B$. 
Here $p,\,r$ are distinct odd primes, and $i,\xi,\eta,\zeta,\mu$ denote primitive $4$th, $8$th, $16$th, $p$th, and $r$th roots of~$1$, in that order.
The numbers of non-equivalent choices for these roots are $2,4,8,2,$ and $2$, respectively.
We denote $D_r$ the dihedral group of order $2r$.
The last entry ($\#$) records the number of equivalence classes of categories for {\em each choice} of the root of $1$.}
\end{table}
 It is well known that non-degenerate braided fusion categories of dimensions $p$ and $pr$ are pointed,
as are those of dimension $p^2$.  By Proposition~\ref{Prop: pointed x ising}, every
reductive category of $\FPdim(\B) =p^3$ must be pointed.
A non-reductive category of this dimension is a $\mathbb{Z}/p\mathbb{Z}$-gauging of $\Rep(\mathbb{Z}/p\mathbb{Z})$ and is, therefore, also pointed.

The only non-pointed braided categories of $\FPdim =4$ are the Ising categories, which are classified in \cite[Appendix B]{DGNO2}.

By Proposition~\ref{Prop: pointed x ising}, a non-pointed {\em reductive} category $\B$ such that $\FPdim(\B) =8\, \text{or}\, 16$ is a product of an anisotropic
pointed category and an Ising category.  The pointed factor cannot contain exactly one fermion, for otherwise 
$\B$ would contain a unique non-trivial Tannakian subcategory $\Rep(\mathbb{Z}/2\mathbb{Z})$. This leads to the following  possibilities:
\[
\B \cong \mathcal{I}_{i\eta} \bt \C(\mathbb{Z}/2\mathbb{Z},\, q_{-i}) \cong \mathcal{I}_{-i\eta} \bt \C(\mathbb{Z}/2\mathbb{Z},\, q_i) \qquad  {\rm   when\ } \FPdim(\B) =8
\]
and 
\[
\B \cong \mathcal{I}_{\eta} \bt   \C(\mathbb{Z}/2\mathbb{Z} \times \mathbb{Z}/2\mathbb{Z} ,h) \cong  \mathcal{I}_{-\eta} \bt   \C(\mathbb{Z}/2\mathbb{Z} \times \mathbb{Z}/2\mathbb{Z} ,e)  \qquad  {\rm   when \ } \FPdim(\B) =16.
\]
For the categories of dimension $16$ with the radical $\Rep(\mathbb{Z}/2\mathbb{Z})$, the mantle must either be non-pointed (and, hence, Ising) or pointed
without Lagrangian subcategories fixed by the $\mathbb{Z}/2\mathbb{Z}$-action (since otherwise, it is the center of a pointed category of dimension $4$ and, hence,
is  itself pointed). In the former case, any Ising category $\mathcal{I}_\eta$ can appear as the mantle, and since $\Aut^{br}(\mathcal{I}_\eta)=1$, the homomorphism from
$\mathbb{Z}/2\mathbb{Z}$ is trivial.
In the latter case, there are four different metric group possibilities for such a  mantle $\B$, listed in the table.  For each of these there is a unique (up to conjugation) non-trivial homomorphism  $\mathbb{Z}/2\mathbb{Z}\to \Aut^{br}(\B)$. 
In all cases,  $H^2(\mathbb{Z}/2\mathbb{Z}, \Inv(\B)) =\mathbb{Z}/2\mathbb{Z}$ and so
there are two monoidal functors $\mathbb{Z}/2\mathbb{Z}\to \uuPic(\B)$. Each of these functors has exactly two lifts to a monoidal $2$-functor
since $H^4(\mathbb{Z}/2\mathbb{Z},\, k^\times)=0$ and $H^3(\mathbb{Z}/2\mathbb{Z},\, k^\times)=\mathbb{Z}/2\mathbb{Z}$. Thus, there are four
monoidal $2$-functors in each case. We exclude the trivial $2$-functor $\mathbb{Z}/2\mathbb{Z} \to \Mod(\mathcal{I}_\eta)$ as it is Tannakian 
imprimitive (see Definition~\ref{def primitive functor}). Resulting  gaugings can be realized as tensor products involving Ising categories. 
In particular, they are non-integral.

For all remaining dimensions, non-pointed categories cannot be reductive.  Furthermore, if $\B\neq \Vect$, the homomorphism $G \to \Aut^{br}(\B)$ 
must be non-trivial. This is impossible when $|G|$ is odd and $\B$ is pointed cyclic, so the above
observation eliminates dimensions $p^4$, $2p^2r$, $2pr^2$, and $p^3r$, cf.\cite{DoN}.

For dimension $8r$, the radical must be $\Rep(\mathbb{Z}/2\mathbb{Z})$ and so $\B= \C(\mathbb{Z}/2r\mathbb{Z},q_{i\mu})$. 
Again, the cohomology group count gives four monoidal $2$-functors, which agrees with the result of \cite{BPR}.

For dimension $4r^2$, either $\B= \Vect$ or $\Inv(\B)= \mathbb{Z}/r\mathbb{Z} \times \mathbb{Z}/r\mathbb{Z}$. In the former case, the category
has a unique Lagrangian subcategory equivalent to $\Rep(D_r)$, and so is equivalent to $\Rep(D^\omega(D_r))$. 
There are four classes of braided categories with these data.
The condition on $\omega$ follows from Proposition~\ref{Lagrangians in ZVecGw}. 
The latter case is partitioned into several subcases, depending on the quadratic form on $\mathbb{Z}/r\mathbb{Z} \times\mathbb{Z}/r\mathbb{Z}$ and the action of $\mathbb{Z}/2\mathbb{Z}$. Since 
\[
\Aut^{br}(\C(\mathbb{Z}/r\mathbb{Z} \times \mathbb{Z}/r\mathbb{Z} ,h))\cong D_{r-1} 
\quad \text{and} \quad
\Aut^{br}(\C(\mathbb{Z}/r\mathbb{Z} \times \mathbb{Z}/r\mathbb{Z} ,e))= D_{r+1}, 
\]
there are precisely three conjugacy classes of such actions: the half-turn rotation  and two classes of reflections.   
Each action gives rise to monoidal $2$-functors. For the hyperbolic quadratic form on $\Inv(\B)$,
the gaugings corresponding to the rotation action give four remaining  twisted doubles $\Rep(D^\omega(D_r))$. Our analysis here gives another proof  of one of the
results of \cite{MS} stating that there are exactly six braided equivalence classes of such doubles.
The reflection actions  give rise to the  centers  of non-integral fusion categories of dimension $2r$; the latter are necessarily Tambara-Yamagami categories \cite{EGO}. 
The count of categories with the elliptic form on $\Inv(\B)$ is similar.


For categories of dimension $4pr$ the only possibility for the mantle is
\[
\B= \C(\mathbb{Z}/pr\mathbb{Z}, q_{\zeta\mu}) \cong \C(\mathbb{Z}/p\mathbb{Z}, q_{\zeta}) \bt \C(\mathbb{Z}/r\mathbb{Z}, q_{\mu}).
\]
Then $\Aut^{br}(\B)= \mathbb{Z}/2\mathbb{Z} \times  \mathbb{Z}/2\mathbb{Z}$ and there are three  possible actions, each leading to two monoidal $2$-functors.
Our count of metaplectic categories agrees with that of   \cite{ACRW, BPR}.

 For categories of dimension $p^2r^2$ we must have $\Inv(\B)= \mathbb{Z}/r\mathbb{Z} \times  \mathbb{Z}/r\mathbb{Z}$ and $G= \mathbb{Z}/p\mathbb{Z}$
since otherwise any homomorphism $G\to \Aut^{br}(\B)$ is trivial. So there are two choices for $\B$, corresponding to the hyperbolic and elliptic forms.
In each case $p$ must divide the order of the corresponding dihedral group, in which case there is a unique, up to conjugation, non-trivial homomorphism 
$\mathbb{Z}/p\mathbb{Z} \to \Aut^{br}(\B)$. It lifts to a unique action, which, in turn, lifts to $p$ monoidal $2$-functors. In the hyperbolic case this recovers
the classification of twisted doubles of $\mathbb{Z}/r\mathbb{Z}\rtimes \mathbb{Z}/p\mathbb{Z}$ from \cite{MS}. Group-theoretical properties of categories of dimension $p^2r^2$ were analyzed in \cite{BCHKNNPR}. 

Finally, dimension $36$ is listed in a separate line since $p$ is assumed odd, but it is really a part of the series in the previous line. These gaugings were studied in \cite{BBCW}. The pointed braided category $\C(\mathbb{Z}/2\mathbb{Z} \times \mathbb{Z}/2\mathbb{Z} ,e)$ is referred by physicists as 
``the three-fermion model."
\end{example}

\bibliographystyle{ams-alpha}

\end{document}